\documentclass{article}
\usepackage{fullpage}

\usepackage{amsmath,amsfonts,amsthm,amssymb}
\usepackage{mathrsfs}
\usepackage{color,xcolor}
\usepackage{ifpdf}
\usepackage{psfrag}
\usepackage{graphicx,graphics,subfigure,epsfig}
\usepackage{url}
\usepackage{hyperref}
\usepackage{xspace}

\newtheorem{theorem}{Theorem}[section]
\newtheorem{corrolary}[theorem]{Corollary}
\newtheorem{lemma}[theorem]{Lemma}

\newtheorem{proposition}[theorem]{Proposition}
\newtheorem{remark}[theorem]{Remark}

\newtheorem{definition}[theorem]{Definition}

\newcommand{\dpar}[2]{\dfrac{\partial #1}{\partial #2}}

 \newcommand{\R}{\mathbb R}
 \newcommand{\Z}{\mathbb Z}

 \newcommand{\Q}{\mathbb Q}
 
\renewcommand{\P}{\mathbb P}

\newcommand{\PP}{\mathbb P}

\newcommand{\KK}{\mathcal K}
\newcommand{\Id}{\text{Id }}

\newcommand{\bbf}{{\mathbf {f}}}
\newcommand{\bn}{{\mathbf {n}}}

\newcommand{\bu}{\mathbf{u}}
\newcommand{\bv}{\mathbf{v}}
\newcommand{\bw}{\mathbf{w}}

\newcommand{\bU}{\mathbf{U}}

\newcommand{\bF}{\mathbf{f}}
\newcommand{\bbF}{\mathbf{\mathcal{F}}}

\newcommand{\bbu}{\mathbf{u}}
\newcommand{\bbg}{\mathbf{g}}

\newcommand{\hbbf}{\hat{\mathbf{f}}}
\newcommand{\hf}{\hat{\mathbf{f}}}

\newcommand{\bx}{\mathbf{x}}
\newcommand{\by}{\mathbf{y}}

\newcommand{\normal}{\mathbf{N}}

\hypersetup{
pdftitle={A general framework to construct schemes that satisfy additional conservation relations.\\
Application to entropy conservative and entropy dissipative schemes}
pdfauthor={R. Abgrall},
pdfpagemode=UseOutlines,
linkbordercolor=0 0 0,
linkcolor=red,
citecolor=blue,
colorlinks=true,
bookmarks = true
}

\begin{document}
\title{The notion of conservation for residual distribution schemes (or fluctuation splitting schemes), with some applications}
\author{R\'emi Abgrall}
\date{}
\maketitle
\begin{abstract}
In this paper, we discuss the notion of discrete conservation for hyperbolic conservation laws. We introduce what we call a  fluctuation splitting schemes (or residual distribution, also RDS) and show on several examples how this cal lead to new development. In particular, we show that most, if not all known schemes can be rephrased in flux form, and also show how to satisfy additional conservation laws. This review paper is built on \cite{AbgrallConservation,Abgrall2017,paola,svetlana,ABGRALL2018640}.

\bigskip

This paper is also a direct consequence of  the work of P.L. Roe, in particular \cite{roe1981,DeconinckRoeStruijs} where the notion of conservation I will discussed is first introduced. In \cite{myway}, P.L. Roe mentions the Hermes project, and the role of Dassault Aviation in it. I was suggested by Bruno Stoufflet, now Vice-President R\&D and advanced business in this company, to have a detailed look at \cite{DeconinckRoeStruijs}. To be honnest, at the time, I did not understood anything, and this was the case for several years. I was lucky to work with Katherine Mer, at the time a postdoc, now research engineer at CEA, and she helped me a lot in starting to understand this notion of conservation. The present contribution can be seen as what I managed to understand after many years playing around the very productive notion of residual distribution schemes (or fluctuation splitting schemes), introduced by P. L. Roe.
\end{abstract}
\section{Introduction}
The aim of this paper is to discuss some aspects related to the notion of weak solutions of 
\begin{subequations}
  \label{eq:1}
\begin{equation}
  \label{eq:1:1}
  \dpar{\bbu}{t}+\text{ div }\bbf(\bbu)=0 \text{ for }\bx\in \R^d
\end{equation}
\begin{equation}
  \label{eq:1:2}
  u(\bx,0)=u_0(\bx), \bx\in \R^d.
\end{equation}
\end{subequations}
Since at least the work of P. Lax, we know that the correct setting to define the notion of solution to \eqref{eq:1} is the following: $\bbu: \R^d\times \R^+\rightarrow \mathcal{D}\subset \R^m$ is a weak solution of \eqref{eq:1} is $\bu\in L^1(\R^d\times \R^+)^m\cap L^\infty(\R^d\times \R^+)^m$, the initial condition in \eqref{eq:1:2} belongs to $L^1(\R^d)^m\cap L^\infty(\R^d)^m$ and for any $\varphi\in C^1_0(\R^d\times \R^+)$, we have
\begin{equation}
  \label{eq:2}
\int_{\R^d\times \R^+} \bigg ( \dpar{\varphi}{t}(\bx,t)\bbu(\bx,t)+\nabla \varphi(\bx,t)\cdot \bbf(\bbu(\bx,t)) \; d\bx \; dt \bigg )+\int_{\R^d}\varphi(\bx,0)\bbu_0(\bx)\; d\bx=0
   \end{equation}
   In \eqref{eq:1}-\eqref{eq:2}, the flux is $\bbf=(f_1, \ldots , f_d)$; the  functions $f_i$ map the open subset $\mathcal{D}$ of $\R^m$ to $\R^m$, and they are assumed to be $C^1$ for simplicity. 
   
   If one is looking for piecewise $C^1$ solutions, one sees that the solution must satisfy the Rankine-Hugoniot relations. More precisely, if $u$ is defined by $\R^d\times \R^+=\Omega^+\cup\Omega^-$ where the boundary $\Sigma=\Omega^+\cap \Omega^-$ is a regular hypersurface. To make things simple, we assume that $\Sigma=\{ (\bx,t)\in \R^d\times \R^+, \bx=\psi(t)\}$ with $\psi$ $C^1$. Then if 
   $$u(x,t)=\left \{ \begin{array}{ll}
   u^+(\bx,t) & \text{ if } (\bx,t)\in \Omega^+\\
   u^-(\bx,t) & \text{ if } (\bx,t)\in \Omega^-
   \end{array}\right .,
   $$
   with $u^\pm$ smooth in $\Omega^\pm$, 
   $u$ is a weak solution of \eqref{eq:1} if it is a classical solution in the interior of $\Omega^+$ and $\Omega^-$, and for any point $(\bx,t)\in \Sigma$, we have
   \begin{equation}
   \label{RH}
   \bbf(u^+(\bx,t))\cdot\bn-\bbf(u^-(\bx,t))\cdot\bn =\sigma (\bx,t) (u^+(\bx,t)-u^-(\bx,t))
   \end{equation}
   where $\bn$ is a normal of $\Sigma$ at $(\bx,t)$, $\bu^\pm(\bx,t)=\lim\limits_{\mathbf{y}\rightarrow \bx, \mathbf{y}\in \Omega^\pm} u^{\pm}(\mathbf{y},t)$ and $\sigma=\dfrac{d\psi}{dt}$.
   
   An important notion is that of entropy. An entropy $E$ is a convex function defined on $\mathcal{D}$ (hence this set is assumed to be convex) such that there exists $\mathbf{g}=(g_1, \ldots , g_d)$ $C^1$ with, for any $j=1, \ldots d$,
   $$\nabla_{\bbu}E^T \nabla_\bbu f_j=\nabla_\bbu g_j.$$
   Hence, if $\bbu$ is $C^1$, we also have  that
   $$\dpar{E}{t}+\text{ div } \mathbf{g}(\bbu) =0.$$
   
   It is well known that the weak solutions of \eqref{eq:1} are not smooth nor continuous in general, so that the above equality cannot be met in general for weak solutions. It is said that a weak solution $\bbu$ is an entropy solution if 
   for any positive $\varphi\in C^1_0(\R^d\times \R^+)$, we have
\begin{equation}
  \label{eq:2:0}
\int_{\R^d\times \R^+} \bigg ( \dpar{\varphi}{t}(\bx,t)E(\bbu)+\nabla \varphi(\bx,t)\cdot \mathbf{g}(\bbu(\bx,t)) \; d\bx \; dt \bigg )+\int_{\R^d}\varphi(\bx,0)E(\bbu_0(\bx))\; d\bx\leq 0.
   \end{equation}
Details can be obtained in classical references such as \cite{RaviartGodelwski,LeVeque}.

The whole purpose is to define a suitable numerical framework for approximating \eqref{eq:1} such that, when a sequence of meshes is considered, with a spatial characteristic size that is converging to zero, the sequence of numerical solution will converge to a weak solution and, if one or more entropies are also considered, to a weak entropy solutions for each of these entropies.

The format of this paper is as follows. I start by recalling the classical notion of discrete conservation introduce by P. Lax and B. Wendroff in the early 60's, and recall what may happen when the approximation does not exactly fit this framework. I also recall that not all scheme fits in that framework, dispites their success.  Then I introduce the notion or residual distribution scheme, and give a Lax-Wendroff like theorem. Using this notion of conservation, I show that any residual distribution scheme is also a finite volume scheme, with non standard flux functions. I also show, using the same concepts, how to satisfy more than one conservation relation, and how to discretise conservative systems not written in conservation form, such as the Euler equation in primitive variables. A conclusion follows.
\section{Classical setting: the Lax Wendroff theorem}
The answer, or one answer to this question, has been given by Lax and Wendroff \cite{LaxWendroff}. We formulate it in one spatial dimension, for simplicity, and provide references for the extension in several dimensions.

\begin{theorem}[Lax-Wendroff]
Consider the problem \eqref{eq:1} for $d=1$.
Consider a mesh $\{x_j=j\Delta x\}_{j\in \Z}$, the control volumes $C_{j}=]x_{j-1/2},x_{j+1/2}]$, and $\lambda>0$
Let $u_j^0$ be an approximation of 
$$\frac{1}{\Delta x}\int_{C_j} u^0(x) dx.$$
Consider the numerical scheme (with $\tfrac{\Delta t}{\Delta x}=\lambda$)
$$u_j^{n+1}=u_j^n-\dfrac{\Delta t}{\Delta x}\big ( \hat{f}_{j+1/2}-\hat{f}_{j-1/2}\big),$$
with the numerical flux: $\hat{f}_{j+1/2}=\hat{f}_{j+1/2}(v_{j-k}, \ldots , v_j, \ldots v_{j+k}).$
We define $u_\Delta $ by:
$$u_\Delta(x,t)=u_j^n \text{ if } (x,t)\in C_{j+1/2}\times [t_n, t_{n+1}[.$$
Assume that:
\begin{enumerate}
\item $\hat{f}_{j+1/2}$ is consistant with $\bbf$:  for any $v$, $\hat{f}_{j+1/2}(v,\ldots, v\ldots v)=\bbf(v)$,
\item $\hat{f}$ is a continuous function of its arguments,
\item The scheme is stable: there exists a constant $C(u_0, f)$ such that $\max\limits_j|u_j^n|\leq C$ for all $n$
\item There exists a subsequence  of $u_\Delta $ that converges to $v$ in $L^2(\R\times \R^+)$
\end{enumerate}
Then $v$ is a weak solution of the problem.
\end{theorem}
\begin{proof} Use the form of the scheme to do integration by part (Abel summation procedure)+Lebesgue's dominated convergence thm.
\end{proof}
\begin{corrolary}
If $E$ is an entropy, and the numerical scheme satisfies the following inequalities
$$E(u_j^{n+1})-E(u_j^n)-\dfrac{\Delta t}{\Delta x}\big ( \hat{g}_{j+1/2}-\hat{g}_{j-1/2}\big)\leq 0$$
where the entropy flux $\hat{g}_{l+1/2}=\hat{g}(v_{j-k}, \ldots , v_j, \ldots v_{j+k})$ is consistent with the entropy flux $\mathbf{g}$, then under the assumptions of the Lax Wendroff theorem, the function $v$ is a weak entropy solution for $E$.
\end{corrolary}
\begin{proof}
The proof is similar, using positive test functions.
\end{proof}
Extension of this results for several spatial dimensions exists, see for example \cite{kroner}.

Since this result, researchers have tried to improve the quality of the numerical approximation by designing more accurate, more robust flux functions, and to encapsulate in better type of time stepping approximation. But in all cases by strickly respecting the Lax Wendroff theorem. Indeed there are many excellent reasons for that, and let us show a couple of counter examples showing what occurs when this framework is violated.

\subsection{Violating the flux form}
Let us consider the Burgers equation,
$$\dpar{u}{t}+u\dpar{u}{x}=0$$ that can also be rewritten in conservation form (if one assumes that the solution is smooth),
$$\dpar{u}{t}+\frac{1}{2}\dpar{u^2}{x}=0, \qquad f(u)=
\dfrac{u^2}{2}.$$
Assuming that $u$ stays positive, two reasonable approximations are
\begin{itemize}
\item Non conservation form:
\begin{equation}\label{Burger:nc}
u_i^{n+1}=u_i^n-\dfrac{\Delta t}{\Delta x} u_i^n\big ( u_i^n-u_{i-1}^n\big ),
\end{equation}
\item Conservation form
\begin{equation}\label{Burger:cons}
u_i^{n+1}=u_i^n-\dfrac{\Delta t}{\Delta x} \big ( f(u_i^n)-f(u_{i-1}^n) \big ).
\end{equation}
\end{itemize}
Both schemes  can be rewritten in a very similar form. In both cases, we have 
$$
u_i^{n+1}=u_i^n-\dfrac{\Delta t}{\Delta x} a_{i-1/2} \big ( u_i^n-u_{i-1}^n\big ).$$
where in the case \eqref{Burger:nc}, $a_{j+1/2}=u_j^n$ and in the case \eqref{Burger:cons}, $a_{j+1/2}=\tfrac{u_{j+1}+u_j}{2}$. Only the speed is modified. 

Since we assume that $u_j^n\in [0, A]$, and since we can write in both cases
$$u_i^{n+1}= (1-\lambda a_{i-1/2})u_i^n+\lambda a_{i-1/2} u_{i-1}^n,$$ we see that $u_i^{n+1}\in [0, A]$ if $\lambda |a_{i+1/2}|\leq \lambda A\leq 1.$ Similarly, we have easily that if in addition that $u_j^n\in [0, A]$
$\Delta x\sum_j |u_j^n|^2\leq B$ then $\Delta x\sum_j |u_j^{n+1}|^2\leq B$ under the same constraint. We also have
that $\sum_{j\in \Z} |u_{j+1}^{n+1}-u_j^{n+1}|\leq \sum_{j\in \Z} |u_{j+1}^n-u_j^n|$ under the same constraints, i.e. the schemes are both total variation diminishing. using Helly's theorem, we see that in both cases a subsequence in converging in $L^2$ to some function. Now, doing numerical simulations we see that we do not converge to the same solution. The one obtained from \eqref{Burger:nc} is a weak solution, thanks to Lax Wendroff theorem. If $E$ is any entropy, denoting by $v$ the gradient of the entropy with respect to $u$, we see that for the conservative case, we get:
$$
v_i^n\big ( u_i^{n+1}-u_i^n\big ) +\lambda \big ( \hf_{i+1/2}-\hf_{i-1/2}\big )=0, \qquad \hf_{j+1/2}=\frac{(u_i^n)^2}{2}.
$$
The entropy flux (when we express the flux in term of the entropy variable) is with some abuse of language
$$g=\int_u v f_v \; dv= v \; f-\int_v f \; dv= v \; f-\theta(v), $$ and following Tadmor, we introduce the entropy flux
$$\hat{g}_{i+1/2}= \bar v_{i+1/2}\hf_{i+1/2}-\bar \theta_{i+1/2}$$
where $\bar f$ is the arithmetic average between $f_i$ and $f_{i+1}$.
After some calculations, we get
\begin{equation*}
\begin{split}
v_i^n\big ( u_i^{n+1}-u_i^n\big ) +\lambda \big ( \hat{g}_{i+1/2}-\hat{g}_{i-1/2}\big ) = \lambda \int_{v_i}^{v_{i+1}} \big ( f(v)-\hat{f}_{i+1/2} \big ) dv+\lambda \int_{v_{i-1}}^{v_{i}} \big ( f(v)-\hat{f}_{i-11/2} \big ) dv.
\end{split}
\end{equation*}
Since
\begin{equation*}
\begin{split}
E_i^{n+1}&=E_i^n+v_i^n(u_i^{n+1}-u_i^n)+\frac{1}{2}\int_{0}^1 E_{uu}((1-s) u_i^n+s u_{i}^{n+1}) (u_i^{n+1}-u_i^n)^2 ds\\
& =E_i^n+v_i^n(u_i^{n+1}-u_i^n)+\frac{\lambda^2}{8}\int_{0}^1 E_{uu}((1-s) u_i^n+s u_{i}^{n+1}) \bigg (\big (u_i^{n}\big )^2-\big (u_{i-1}^n\big )^2\bigg )^2 ds\\
&=E_i^n+v_i^n(u_i^{n+1}-u_i^n)+\frac{\lambda^2}{2} a_{i-1/2}^2 (u_i^n-u_{i-1}^n)^2 \; \int_{0}^1 E_{uu}((1-s) u_i^n+s u_{i}^{n+1}) ds
\end{split}
\end{equation*}
we get
\begin{equation*}
\begin{split}
E_i^{n+1}-E_i^n&+\lambda \big (\hat{g}_{i+1/2}-\hat{g}_{i-1/2} \big ) =\lambda \int_{v_i}^{v_{i+1}} \big ( f(v)-\hat{f}_{i+1/2} \big ) dv+\lambda \int_{v_{i-1}}^{v_{i}} \big ( f(v)-\hat{f}_{i-1/2} \big ) dv\\
& \qquad +\frac{\lambda^2}{2} a_{i-1/2}^2 (u_i^n-u_{i-1}^n)^2  \; \int_{0}^1 E_{uu}((1-s) u_i^n+s u_{i}^{n+1}) ds
\end{split}
\end{equation*}
In our case, we always have $a_{j+1/2}>0$, since $u_i^{n}>0$ for all $i$ and $n$, we see that for $\lambda$ small enough the right hand side is negative because
$$\int_{v_i}^{v_{i+1}} \big ( f(v)-\hat{f}_{i+1/2} \big ) dv<0, \qquad \int_{v_{i-1}}^{v_{i}} \big ( f(v)-\hat{f}_{i-1/2} \big ) dv<0.$$
So the weak solution of the conservative scheme is an entropy solution for any entropy, and by uniqueness, this is the solution.
\begin{figure}[h]
\begin{center}
\includegraphics[width=0.5\textwidth]{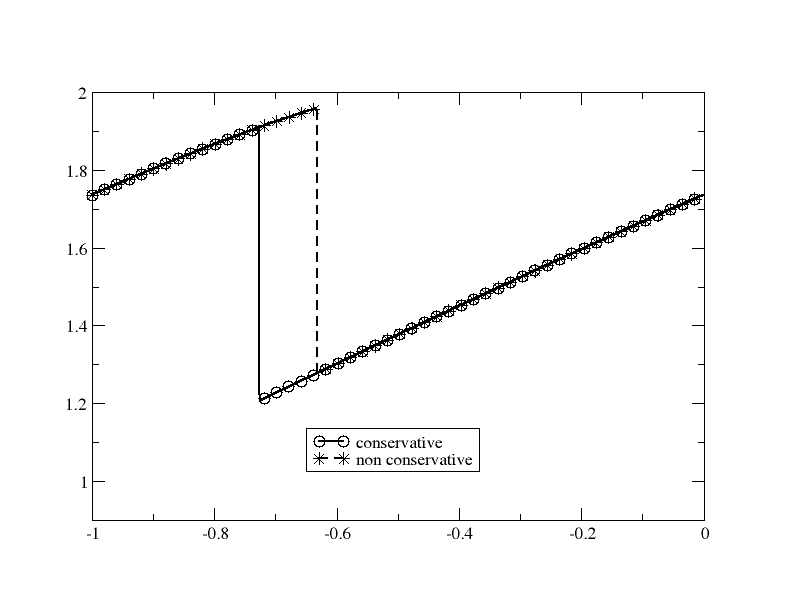}
\end{center}
\caption{\label{cons:nc} Solution of the Burgers equation with $u_0(x)=1+\cos(2\pi(x+1/2))$ with periodic boundary conditions.}
\end{figure}

In \cite{hou}, it is shown that a scheme that is stable in $L^\infty$ and in $BV$ and for which there exists a subsequence  that converges in  $L^1_{loc}$ to a solution, in the sense of distribution, of
$$\dpar{u}{t}+\dpar{f(u)}{x}=\mu$$
where $\mu$ is a locally bounded real-valued Borel measure defined on $\R\times \R^+$. In the case that the initial condition  has a finite number of monotonicity changes, they also show that $\mu$ is concentrated on the union of the curves of discontinuity, at least for small times.  It seems quite difficult to provide a more constructive description of $\mu$ and this does not prevent that $\mu \equiv 0$\ldots ! Though it seems to be seldom the case.

However, there are many cases where one would like to deal with the non conservative version of a model, see bellow for a practical example.

\subsection{Change of variables}
It is well know that nonlinear change of variables are not permitted. The classical example is again the Burgers equation
$$\dpar{u}{t}+\frac{1}{2}\dpar{u^2}{x}=0.$$
If one sets $u=v^3$, this is a one-to-one change of variable, and the smooth solutions in $v$ will satisfy
$$\dpar{v}{t}+\frac{1}{4}\dpar{v^4}{x}=0.$$
However the weak solutions are not similar, because the Rankine-Hugoniot relations are different:
For the original Burgers equation, the velocity of a shock between the states $u^\pm$ is
$$\sigma=\dfrac{u^++u^-}{2},$$
while for the second problem it is
$$\sigma'=\dfrac{1}{4}\dfrac{(v^+)^4-(v^-)^4}{v^+-v^-}=\dfrac{(v^+)^3+3 (v^+)^2v^-+3v^+(v^-)^2+(v^-)^3}{4},$$
and even if $u^\pm=(v^\pm)^3$, there is no chance that the two velocity match, in general.

\subsection{Some questions}
However, in many case, one would like to change variables and/or work with non conservative models. A good example is the the Euler system for fluid mechanics.

The conserved version of this system is, setting $\bu=(\rho, \mathbf{m}=\rho \mathbf{v}, E)^T$ where here $\mathbf{v}$ represents the velocity, $\rho$ is the density and $E=e+\tfrac{1}{2}\rho \mathbf{v}^2$ is the total energy, with $e$ the internal energy, the flux is
$$\bbf(\bu)=\begin{pmatrix} \mathbf{m}\\ \dfrac{\mathbf{m}\otimes\mathbf{m}}{\rho}+p\text{Id}\\ \mathbf{m}\dfrac{E+p}{\rho}\end{pmatrix}.$$
The system is closed if the pressure is given in term of $\bu$, $p=p(e,\rho)$. Defining the enthalpy $H=\tfrac{E+p}{\rho}$, under the condition that 
$$\dpar{p}{e} (H-\frac{1}{2}\mathbf{v}^2)+\dpar{p}{\rho}>0,$$ the system is hyperbolic. The speed of sound is given by
$$a^2=\dpar{p}{e} (H-\frac{1}{2}\mathbf{v}^2)+\dpar{p}{\rho}.$$ If $p=(\gamma-1)e$, i.e. for a perfect gas, the speed of sound is given by the classical
$$a^2=\gamma\dfrac{p}{\rho}.$$

Another way to write the system is:
\begin{equation}\label{euler:nc}
\dpar{}{t}\begin{pmatrix}\rho\\ \rho \bv \\ p\end{pmatrix} +\begin{pmatrix}\text{ div } \rho \bv\\
\text{div }(\rho \bv\otimes \bv+p\,\text{Id})\\
\bv\cdot \nabla p+\rho a^2 \text{ div }\bv \end{pmatrix}
=0.
\end{equation}
and this form is interesting because the pressure is a variable that can be tracked directly from the system, and not as a consequene of the total energy equation. Unfortunately the system \eqref{euler:nc} is not in conservation form, and then starting from that relation is a priori not a good idea.

\subsection{Many schemes are not naturally written in a finite volume form}
To solve the problem \eqref{eq:1}, there are other methods than finite volume schemes, this does not mean they are not good. However, it seems that one of the essential feature of finite volume schemes, i.e. local conservation, is lost. In order to simplify, we consider the steady problem
\begin{subequations}\label{eq1}
\begin{equation}
\label{eq1:1}
\text{ div }\bbf(\bu)=0\qquad \text{in }\Omega
\end{equation}
subjected to
\begin{equation}\label{eq1:2}
\min( \nabla_\bu \bbf(\bu)\cdot \bn(\bx), 0) (\bu-\bu_b)=0 \text{ on }\partial\Omega.
\end{equation}
\end{subequations}
The domain $\Omega$ is assumed to be bounded, and regular. We assume for simplicity that its boundary is never characteristic. We also assume that it has a polygonal shape and thus any triangulation that we consider covers $\Omega$ exactly, only for simplicity.
In \eqref{eq1:2}, $\bn(\bx)$ is the outward unit vector at $\bx\in \partial \Omega$ 
and $\bu_b$ is a  regular enough function.  The weak formulation of \eqref{eq1} is: $\bu\in  L^\infty(\Omega)$ is a weak solution of \eqref{eq1} if for any $\varphi\in C^1_0(\Omega)$, 
\begin{equation}
\label{weak:eq1}
-\int_\Omega \nabla \bv\cdot \bbf(\bu_h)\; d\bx+\int_{\partial \Omega} \bv \big ( \mathbf{\mathcal{F}}_\bn(\bu,\bu_b)-\bbf(\bu)\cdot\bn\big ) \; d\gamma=0
\end{equation}
where $\mathbf{\mathcal{F}}_\bn$ is a flux that is almost everywhere the upwind flux:
$$\mathbf{\mathcal{F}}_\bn(\bu,\bu_b)=\left \{\begin{array}{ll}
\bbf(\bu_b)\cdot\bn & \text{ if } \nabla_\bu \bbf(\bu)\cdot \bn >0\\
\bbf(\bu)\cdot\bn & \text{ else.}
\end{array}
\right .
$$
Let us show some examples, and before let us introduce some notations.

 We denote by $\mathcal{E}_h$ the set of internal edges/faces of $\mathcal{T}_h$, and by $\mathcal{F}_h$ those contained in $\partial \Omega$.  $\KK$ stands either for an element $K$ or a face/edge $e\in \mathcal{E}_h\cup \mathcal{F}_h$. The boundary faces/edges are denoted by $\Gamma$.  The mesh is assumed to be shape regular, $h_K$ represents the diameter of the element $K$. Similarly, if $e\in \mathcal{E}_h\cup \mathcal{F}_h$, $h_e$ represents its diameter.

We follow Ciarlet's definition \cite{ciarlet,ErnGuermond} of a finite element approximation: we have a set of degrees of freedom $\Sigma_K$ of linear forms acting on the set $\PP^k$ of polynomials of degree $k$ such that the linear mapping
 $$q\in \PP^k\mapsto \big (\sigma_1(q), \ldots, \sigma_{|\Sigma_K|}(q)\big )$$
 is one-to-one. The space $\PP^k$ is spanned by the basis function $\{\varphi_{\sigma}\}_{\sigma\in \Sigma_K}$  defined by
 $$\forall \sigma,\, \sigma',  \sigma(\varphi_{\sigma'})=\delta_\sigma^{\sigma'}.$$
  We have in mind either Lagrange interpolations where the degrees of freedom are associated to points in $K$, or other type of polynomials approximation such as B\'ezier polynomials where we will also do the same geometrical identification.
 Considering all the elements covering $\Omega$, the set of degrees of freedom is denoted by $\mathcal{S}$ and a generic degree of freedom  by $\sigma$. We note that for any $K$, 
 $$\forall \bx\in K, \quad \sum\limits_{\sigma\in K}\varphi_\sigma(\bx)=1.$$
 For any element $K$, $\#K$ is the number of degrees of freedom in $K$. If $\Gamma$ is a face or a boundary element, $\#\Gamma$ is also the number of degrees of freedom in $\Gamma$.

 The integer $k$ is assumed to be the same for any element.  We define 
$$\mathcal{V}^h=\bigoplus_K\{ \bv\in L^2(K), \bv_{|K}\in \PP^k\}.$$
The solution will be sought for in a  space $v_h$ that is:
\begin{itemize}
\item Either $v_h=\mathcal{V}^h$. In that case, the elements of $v_h$ can be discontinuous across internal faces/edges of $\mathcal{T}_h$. There is no conformity requirement on the mesh.
\item Or  $v_h=\mathcal{V}_h\cap C^0(\Omega)$ in which case the mesh needs to be conformal.
\end{itemize}

We also need to integrate functions. This is done via quadrature formula, and the symbol $\oint$ used in volume integrals
$$\oint_K v(\bx)\; d\bx$$
or boundary integrals
$$
\oint_{\partial K} v(\bx)\; d\gamma$$
means that these integrals are done via user defined numerical quadratures.

 If $e\in \mathcal{E}_h$, represents any  \emph{internal} edge, i.e. $e\subset K\cap K^+$ for two elements $K$ and $K^+$,  we define for any function $\psi$ the jump  $[\nabla \psi ]=\nabla \psi_{|K}-\nabla \psi_{| K^+}$. Here the choice of $K$ and $K^+$ is important and will become clear in each example. Similarly, $\{\bv\}=\tfrac{1}{2}\big (\bv_{|K}+\bv_{|K^+}\big )$.
 
 If $\mathbf{x}$ and $\mathbf{y}$ are two vectors of $\R^q$, for $q$ integer, $\langle \bx,\by\rangle$ is their scalar product. In some occasions, it can also be denoted as $\bx\cdot\by$ or $\bx^T\by$.  We also use  $\bx\cdot\by$ when $\bx$ is a matrix and $\by$ a vector: it is simply the matrix-vector multiplication.


\bigskip

The first example is the SUPG scheme, originally designe by T. J. Hughes and collaborators. A variant of it is used as a production code in Dassault, where shocks need to be considered \ldots
The formulation is: find $\bu_h\in v_h=\mathcal{V}_h\cap C^0(\Omega)$ such that for any $\bw_h\in v_h$, 
\begin{equation}\label{SUPG}
\begin{split}
a(\bu_h,\bw_h)&:=-\int_\Omega \nabla \bw_h\cdot \bF(\bu_h)\; d\bx+\sum\limits_{K\subset \Omega}h_K\int_K\big [ \nabla \bF(\bu_h)\cdot\nabla \bw_h\big ] \; \tau_K\; \big [\nabla\bF(\bu_h)\cdot \nabla \bu_h\big ] d\bx\\
&\qquad +\int_{\partial \Omega} \bw_h \big (\bbF_\bn(\bu_h,\bu_b)-\bF(\bu_h)\cdot\bn\big ) \;d\gamma . 
\end{split}
\end{equation}
here $\tau_K$ is a strictly positive parameter, and it specific design is at the core of the method (for stability reasons). In \eqref{SUPG}, the first term of the right hand side is the Galerkin term. It is obtained by multiplying \eqref{eq1} by the test function $\bw_h$, apply the divergence theorem and take into account the continuity of the elements of $v_h$ accross the edges of the mesh. The last term corresponds to the boundary conditions. The second term is a stability term: if $\bw_h=\bu_h$, this term is positive, but cancel if $\nabla\bF(\bu_h)\cdot \nabla \bu_h=0$, so that the residual property is met: the Galerkin term also vanishes (up-to the boundary terms), if $\nabla\bF(\bu_h)\cdot \nabla \bu_h=0$.

A variant of the SUPG scheme  is obtained when one changes the stabilisation term:
\begin{equation}\label{Jump}
\begin{split}
a(\bu_h,\bw_h)&:=-\int_\Omega \nabla \bw_h\cdot \bF(\bu_h)\; d\bx+\sum\limits_{e \subset \Omega}\theta_e h_K^2\int_e \big [ \nabla \bw_h \big ]\cdot \big [ \nabla \bu_h\big ] \; d\gamma \\
&\qquad +\int_{\partial \Omega} \bw_h \big (\bbF_\bn(\bu_h,\bu_b)-\bF(\bu_h)\cdot\bn\big ) \; d\gamma . \qquad \theta_e>0
\end{split}
\end{equation}
Here, the residual property is kept, up to boundary terms, for smooth solutions, since the jump term cancels, see \cite{burman} for details.

In these two cases, there is no clear flux formulation. One can also say the same for the discontinuous Galerkin schemes,
where we look for $\bu_h, \bv_h\in v_h=\mathcal{V}^h$ such that
\begin{equation}\label{DG:var}
a(\bu_h,\bv_h):=\sum\limits_{K\subset \Omega}\bigg ( -\int_K\nabla\bv_h\cdot\bbf(\bu_h) d\bx+\int_{\partial K}\bv_h\cdot \hat{\bbf}_\bn(\bu^{h},\bu^{h,-}) \;d\gamma \bigg ).
\end{equation}

 If a flux formulation is obvious for the averages of the conserved variables, this corresponds to only one degree of freedom. One can easily construct isomorphisms between 
$\PP^k$ and $\R^{\text{dim }\PP^k}$: for this, one can split the elements $K$ into $\text{dim }\PP^k$ non overlapping control volumes and in general the mapping between $\PP^k$ and the $\text{dim }\PP^k$ dimensional vector consisting of the average of the elements of $\PP^k$ on these control volume will be one-to one. This means that one can reformulate the discontinuous Galerkin scheme as a scheme actioning not on $\mathcal{V}^h$ but on a direct sum of copies of $\R^{\text{dim }\PP^k}$. One one hand one would expect a natural finite volume formulation of the method, from geometry, on the other side, this formulation is not clear, though this kind of idea have already be used for hexahedral meshes in \cite{zbMATH06725695} or the DGSEM schemes, see for example \cite{zbMATH06798170}.

\section{A different point of view}
Looking again at the schemes \eqref{Jump} and \eqref{SUPG}, we see that we can rewrite them as:
\begin{equation}
\label{RDS}
\sum_{K| \sigma\in K}\Phi_\sigma^K(\bbu_h)+\sum_{f\subset\partial \Omega|\sigma\in f}\Psi_\sigma^f(\bbu_h)=0
\end{equation}
where the element and boundary residuals $\Phi_\sigma^K(\bbu_h)$ and $\Psi_\sigma^f(\bbu_h)$ are defined as follows:
\begin{itemize}
\item Case of the SUPG scheme \eqref{SUPG}
 \begin{equation*}
    \begin{split}\Phi_\sigma^K(\bu_h)&=\int_{\partial K}\varphi_\sigma \bF(\bu_h)\cdot \bn \; d\gamma -\int_K \nabla \varphi_\sigma\cdot \bF(\bu_h) \; d\bx\\&+h_K
    \int_K \bigg (\nabla_\bu\bF(\bu_h)\cdot \nabla \varphi_\sigma \bigg )\tau_K \bigg (\nabla_\bu\bF(\bu_h)\cdot \nabla \bu_h \bigg )\;d\bx
    \end{split}
    \end{equation*}
    with $\tau_K>0$.
    \item Case of the scheme \eqref{Jump}
          \begin{equation*}\Phi_\sigma^K(\bu_h)=\int_{\partial K}\varphi_\sigma \bF(\bu_h)\cdot \bn\; d\gamma -\int_K \nabla \varphi_\sigma\cdot \bF(\bu_h)\; d\bx +
    \sum\limits_{e \text{ faces of }K} \frac{\theta_e}{2} h_e^2 \int_{\partial K} [\nabla \bu]\cdot [\nabla \varphi_\sigma]\; d\gamma\end{equation*}
    with $\theta_e>0$.
    \item In both cases, 
    $$\Psi_\sigma^f(\bbu_h)=\int_{f} \varphi_h \big (\bbF_\bn(\bu_h,\bu_b)-\bF(\bu_h)\cdot\bn\big ) \; d\gamma .$$
    \end{itemize}
    
    In addition, since $\sum\limits_{\sigma\in K}\varphi_\sigma=1$, we see that in both cases, for the internal elements $K$, we have
    \begin{subequations}\label{RD:cons}
    \begin{equation}
    \label{RD:cons:K}
    \Phi^K(\bbu_h):=\sum_{\sigma\in K} \Phi_\sigma^K(\bbu_h)=\int_{\partial K} \bF(\bu_h)\cdot \bn \; d\gamma,
    \end{equation}
    and for the boundary elements, we have
        \begin{equation}
    \label{RD:cons:f}
    \Psi^f(\bbu_h):=\sum_{\sigma\in K} \Psi_\sigma^f(\bbu_h)=\int_{f} \big (\bbF_\bn(\bu_h,\bu_b)-\bF(\bu_h)\cdot\bn\big ) \; d\gamma
    \end{equation}
    \end{subequations}
    
    These are not the only schemes that can be rewritten in the form \eqref{RDS}. For example, the dG scheme \eqref{DG:var} can be rewritten as such with
    $$\Phi_\sigma^K(\bbu_h)=-\int_K\nabla\varphi_\sigma\cdot\bbf(\bu_h) d\bx+\int_{\partial K}\varphi_\sigma\cdot \hat{\bbf}_\bn(\bu^{h},\bu^{h,-}) \; d\gamma
$$
where \eqref{RD:cons:K} has to be slightly modified into
$$
  \Phi^K(\bbu_h):=\sum_{\sigma\in K} \Phi_\sigma^K(\bbu_h)=\int_{\partial K} \hat{\bbf}_\bn(\bu^{h},\bu^{h,-})\; d\gamma.
  $$

Any finite volume also has a  similar structure.  Let us start with the scheme
$$
|C_j| (u_j^{n+1}-u_j^n)+\Delta t \big (\hat{f}_{j+1/2}-\hat{f}_{j-1/2}\big )=0,$$ 
defined for the mesh $\{ x_j\}_{j\in Z}$. The control volumes are $[x_{j-1/2},x_{j+1/2}]$ with $x_{j+1/2}=\tfrac{x_j+x_{j+1}}{2}$. Here 
$|C_j$ stands for the measure of $C_j$, $|C_j|=x_{j+1/2}-x_{j-1/2}$. As in \cite{myway}, I introduce the quantities defined for the element $K_{j+1/2}=[x_j,x_{j+1}]$ whatever $j\in \Z$ by:
$$
\overrightarrow{\Phi}_{K_{j+1/2}}(u_h)=\hat{f}_{j+1/2}-f(u_j), \qquad \overleftarrow{\Phi}_{K_{j+1/2}}(u_h)=f(u_{j+1})-\hat{f}_{j+1/2},$$
where $u_h$ is the piecewise linear interpolant of the $\{u_l\}_{l\in \Z}$ at the mesh nodes (hence we change interpretation)
we see that the finite volume scheme can be rewritten as 
$$
|C_j| (u_j^{n+1}-u_j^n)+\Delta t \big (\overrightarrow{\Phi}_{K_{j+1/2}}(u_h)+ \overleftarrow{\Phi}_{K_{j-1/2}}(u_h)\big )
$$
In a way, the quantity $\overrightarrow{\Phi}_{K_{j+1/2}}(u_h)$ is the "amount" of information sent by $K_{j+1/2}$ to the vertex $j$, while $\overleftarrow{\Phi}_{K_{j+1/2}}(u_h)$ is the "amount" of information sent by $K_{j+1/2}$ to the vertex $j+1$. Since the vertex $j$ belongs to $K_{j+1/2}$ and $K_{j-1/2}$, we just add the two pieces of information. Now going back to the element $K_{j+1/2}$, adding together the two pieces of informations, we get the total information, i.e.
$$\overrightarrow{\Phi}_{K_{j+1/2}}(u_h)+\overleftarrow{\Phi}_{K_{j+1/2}}(u_h)=f(u_{j+1})-f(u_j)=\int_{\partial K_{j+1/2}} f(u_h)\cdot \bn \; d\gamma$$
with some abuse of language.

\bigskip

This construction can be extended to any volume, and the key fact is that the volumes are closed, so that the integral of the outward unit normal to the control volume vanishes. Let us be more explicit, and we choose a 2D example.

Consider  a conformal mesh, the vertices are $\bx_i$, and the elements are generically denoted by $K$. For simplicity, we assume that $K$ is a simplex, so it is convex and we can consider its centroid. For any face, we consider again the centroid, and we connect all this in the same way as on

Here, we rephrase \cite{Abgrall99}. The notations are defined in Figure \ref{fig:fv}.
\begin{figure}[h]
\begin{center}
\subfigure[]{\includegraphics[width=0.45\textwidth]{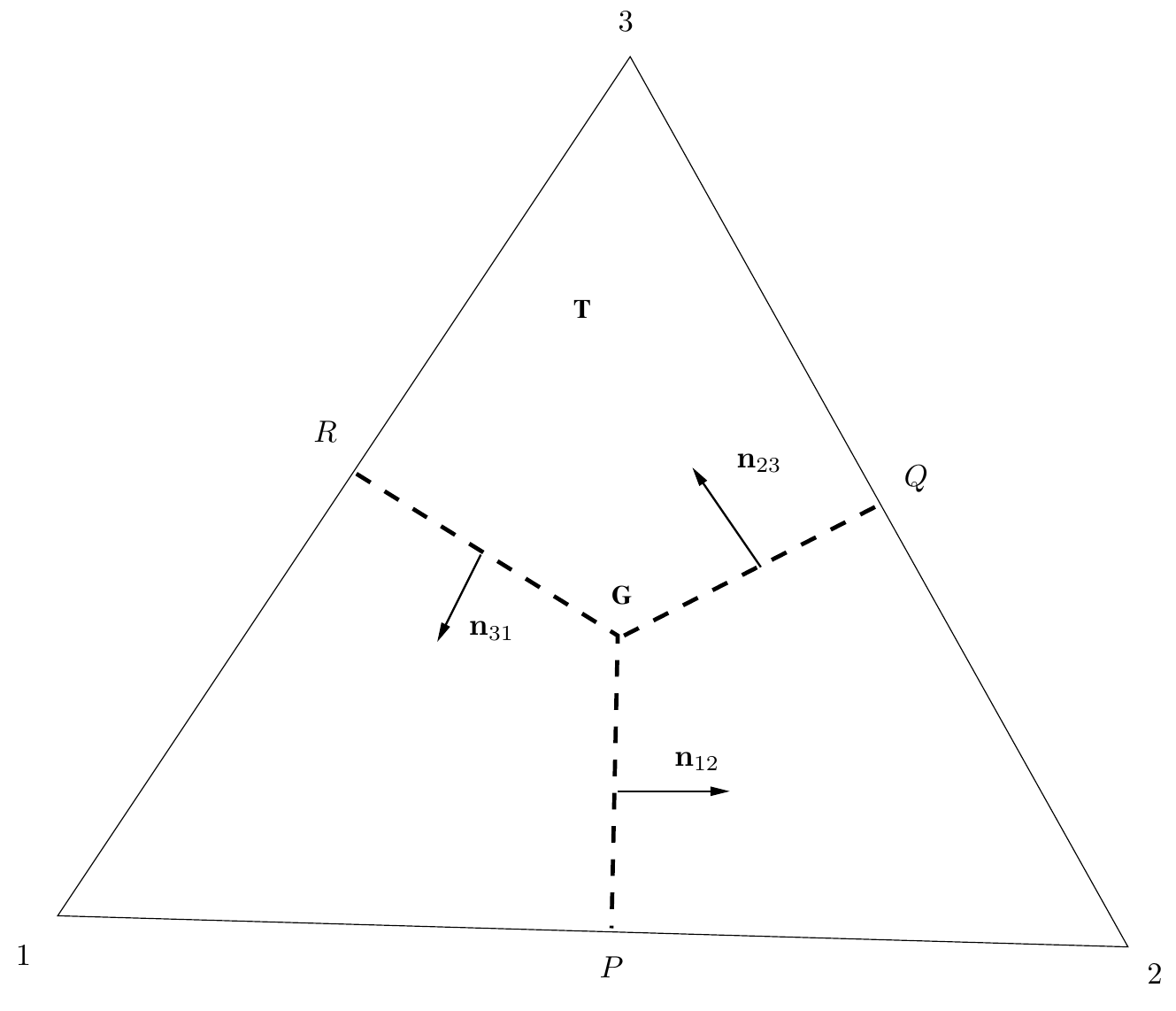}}
\subfigure[]{\includegraphics[width=0.45\textwidth]{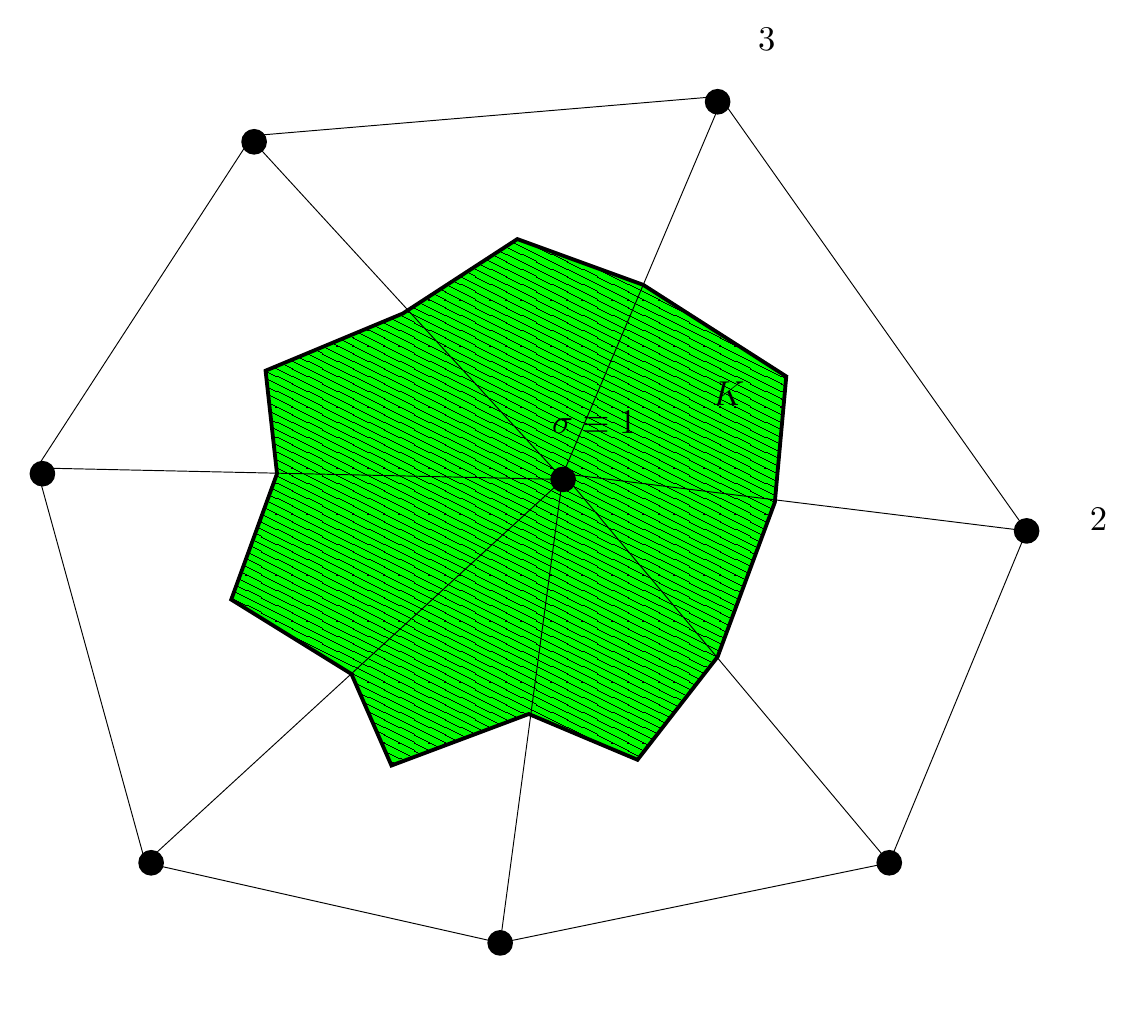}}
\end{center}
\caption{\label{fig:fv} Notations for the finite volume schemes. On the left: definition of the control volume for the degree of freedom $\sigma$.
 The vertex $\sigma$ plays the role of the vertex $1$ on the left picture for the triangle K. The control volume $C_\sigma$ associated to $\sigma=1$ is green on the right and corresponds to $1PGR$ on the left. The vectors $\bn_{ij}$ are normal to the internal edges scaled by the corresponding edge length}
\end{figure}
Again, we specialize ourselves to the case of triangular elements, but  \emph{exactly the same arguments} can be given for more general elements,
 provided a conformal approximation space can be constructed. This is  the case for
triangle elements, and we can take $k=1$.

The control volumes in this case are defined as the median cell, see figure \ref{fig:fv}.
We concentrate on the  approximation of $\text{div } \bbf$, see equation \eqref{eq1}.   Since the boundary of $C_\sigma$ is a closed polygon, the scaled outward normals $\bn_\gamma$ to $\partial C_\sigma$ sum up to 0:
$$
\sum_{\gamma \subset \partial C_\sigma}\bn_\gamma=0$$
where $\gamma$ is any of the segment included in $\partial C_\sigma$, such as $PG$ on Figure \ref{fig:fv}. 
Hence
\begin{equation*}
\begin{split}
\sum_{\gamma \subset \partial C_\sigma} \hbbf_{\bn_\gamma }(\bu_\sigma , \bu^-& )= \sum_{\gamma \subset \partial C_\sigma} \hbbf_{\bn_\gamma }(\bu_\sigma, \bu^- )- \bigg (\sum_{\gamma \subset \partial C_\sigma}\bn_\gamma\bigg )\cdot \bbf (\bu_\sigma)\\
&=\sum\limits_{K, \sigma\in K} \sum\limits_{\gamma \subset \partial C_\sigma\cap K} \big ( \hbbf_{\bn_\gamma }(\bu_\sigma, \bu^- )-\bbf (\bu_\sigma)\cdot \bn_\gamma \big )
\end{split}
\end{equation*}
To make things explicit, in $K$, the internal boundaries are $PG$, $QG$ and $RG$, and those around $\sigma\equiv 1$ are $PG$ and $RG$.
We set
\begin{equation}
\begin{split}
\Phi_\sigma^K(\bu_h)&=\sum\limits_{\gamma\subset \partial C_\sigma\cap K} \big ( \hbbf_{\bn_\gamma }(\bu_\sigma, \bu^- )-\bbf (\bu_\sigma)\cdot \bn_\gamma \big )\\
&=\sum\limits_{\gamma\subset \partial ( C_\sigma\cap K )}  \hbbf_{\bn_\gamma }(\bu_\sigma, \bu^- ).
\end{split}
\label{fv:res:sigma}
\end{equation}
The last relation uses the consistency of the flux and the fact that $C_\sigma\cap K$ is a closed polygon. The quantity $\Phi_\sigma^K(\bu_h)$ is the normal flux on $C_\sigma\cap K$.
If now we sum up these three quantities and get:
\begin{equation*}
\begin{split}
\sum_{\sigma\in K} \Phi_\sigma^K(\bu_h)&= \bigg ( \hbbf_{\bn_{12}}(\bu_1,\bu_2)-\hbbf_{\bn_{13}}(\bu_1,\bu_3)-\bbf(\bu_1)\cdot\bn_{12}+\bbf(\bu_1)\cdot\bn_{31}\bigg )\\
&+\bigg ( \hbbf_{\bn_{23}}(\bu_2,\bu_3)-\hbbf_{\bn_{12}}(\bu_2,\bu_1)+\bbf(\bu_2)\cdot\bn_{12}-\bbf(\bu_2)\cdot\bn_{23}\bigg )\\
&+\bigg ( -\hbbf_{\bn_{23}}(\bu_3,\bu_2)+\hbbf_{\bn_{31}}(\bu_3,\bu_1)-\bbf(\bu_3)\cdot\bn_{23}+\bbf(\bu_3)\cdot\bn_{31}\bigg )\\
&= \bbf(\bu_1)\cdot \big ( \bn_{12}-\bn_{31}\big ) +\bbf(\bu_2)\cdot \big ( -\bn_{23}+\bn_{31}\big )
+\bbf(\bu_3)\cdot \big ( \bn_{31}-\bn_{23}\big )\\
&=\bbf(\bu_1)\cdot\frac{\bn_1}{2}+\bbf(\bu_2)\cdot\frac{\bn_2}{2}+\bbf(\bu_3)\cdot\frac{\bn_3}{2}
\end{split}
\end{equation*}
where $\bn_j$ is the scaled inward normal of the edge opposite to vertex $\sigma_j$, i.e. twice the gradient of the $\PP^1$ basis function
 $\varphi_{\sigma_j}$ associated to this degree of freedom.
Thus, we can reinterpret the sum as the boundary integral of the Lagrange interpolant of the flux.
The finite volume scheme is then a residual distribution scheme with residual defined by \eqref{fv:res:sigma}
and a total residual defined by
\begin{equation}
\label{fv:tot:residu}
\Phi^K:=\int_{\partial K} \bbf^h\cdot \bn , \qquad \bbf^h=\sum_{\sigma\in K} \bbf(\bu_\sigma)\varphi_\sigma.
\end{equation}

\subsection{The residual distribution point of view}

As we said, this point of view was first introduced by P.L. Roe in his seminal 1981 paper, \cite{roe1981}, and then further developped for several dimensions in \cite{DeconinckRoeStruijs}. However, from an historical view point, one of the very first multidimensional residual distribution paper was written by Ni, an engineer at Bombardier, see \cite{Ni:81}.
\begin{definition}[Residual distribution schemes]
Considering \eqref{eq:1}, and a mesh of $\Omega$ made of simplices $K$, we will say that a scheme is a residual distribution scheme if one approximates the solution $\bbu$ of \eqref{eq:1} by $\bbu_h\in v_h$ $v_h$ is the set of functions that are polynomials of degree $k$ on each element $K$ and  globally continuous or not by the scheme \eqref{RDS} where the residuals $\Phi_\sigma^K(\bbu_h)$ and the boundary residuals $\Psi_\sigma^f(\bbu_h)$ satisfy the conservation relations \eqref{RD:cons:K} and \eqref{RD:cons:f}.
\end{definition}

One can show and is a generalisation of the classical Lax-Wendroff theorem, see \cite{AbgrallRoe}.
\begin{theorem}\label{th:LW}
Assume the family of meshes $\mathcal{T}=(\mathcal{T}_h)$
is shape regular. We assume that the residuals $\{\Phi_\sigma^{\mathcal{K}}\}_{\sigma\in \KK}$,
 for $\mathcal{K}$ an element or a boundary element of $\mathcal{T}_h$, satisfy: \begin{itemize}
\item For any $M\in \R^+$, there exists a constant $C$ which depends only on the family of meshes $\mathcal{T}_h$ and $M$ such that
 for any $\bu_h\in v_h$ with $||\bu_h||_{\infty}\leq M$, then
$$\big|\Phi^\KK_\sigma({\bu_h}_{|\KK})\big |\leq C\sum_{\sigma, \sigma'\in \KK}|\bu_\sigma^h-\bu_{\sigma'}^h|$$
\item The conservation relations \eqref{RD:cons:K} and \eqref{RD:cons:f}.
\end{itemize}
Then if there exists a constant $C_{max}$ such that the solutions of the scheme \eqref{RDS} satisfy $||\bu_h||_{\infty}\leq C_{max}$ and a function $\bv\in L^2(\Omega)$ such that $(\bu_h)_{h}$ or at least a sub-sequence converges to $\bv$ in $L^2(\Omega)$, then $\bv$ is a weak solution of \eqref{eq1}
\end{theorem}

An immediate side result is the following result on entropy inequalities:
\begin{proposition}\label{th:entropy}
Let $(U,\mathbf{g})$ be a  entropy-flux couple for \eqref{eq1} and $\hat{\mathbf{g}}_\bn$ be a numerical entropy flux consistent with $\mathbf{g}\cdot \bn$. Assume that the residuals satisfy:
for any element $K$,
\begin{subequations}\label{entropy}
\begin{equation}\label{entropy:1}
\sum_{\sigma \in K}\langle\nabla_\bu U(\bu_\sigma), \Phi_\sigma^K\rangle \geq \int_{\partial K} \hat{\mathbf{g}}_\bn(\bu_h,\bu^{h,-}) \; d\gamma
\end{equation}
and for any boundary edge $e$,
\begin{equation}\label{entropy:2}
\sum_{\sigma \in e}\langle\nabla_\bu U(\bu_\sigma) ,  \Phi_\sigma^e\rangle  \geq \int_{e} \big (\hat{\mathbf{g}}_\bn(\bu_h,\bu_b)- \mathbf{g}(\bu_h)\cdot \bn \big )\; d\gamma.
\end{equation}
\end{subequations}
Then, under the assumptions of theorem \ref{th:LW}, the limit weak solution also satisfies the following entropy inequality: for any $\varphi\in C^1(\overline{\Omega})$, $\varphi\geq 0$, 
$$-\int_\Omega \nabla \varphi\cdot \mathbf{g}(\bu) \; d\bx+\int_{\partial\Omega^-}\varphi\;\mathbf{g}(u_b)\cdot \bn  \; d\gamma \leq 0.$$
\end{proposition}

Instead of considering conservation at the level of the internal faces of the mesh, we consider it at the level of the elements. This opens new perspectives, and we will show some of them in the sequel

\bigskip
\subsection{New examples}
Using this point of view, and following the pionneering work of P.L Roe, one can define
  the limited  Residual Distributive Schemes, see \cite{abgrallLarat,abgralldeSantisSISC}, namely
    \begin{equation}
    \label{schema RDS}\Phi_\sigma^K(\bu_h)=\beta_\sigma \int_{\partial K}\bF(\bu_h)\cdot \bn\; d\gamma
    \end{equation}
    or 
    \begin{equation}
    \label{schema RDS SUPG}\Phi_\sigma^K(\bu_h)=\beta_\sigma \int_{\partial K}\bF(\bu_h)\cdot \bn\; d\gamma+h_K
    \int_K \bigg (\nabla_\bu\bF(\bu_h)\cdot \nabla \varphi_\sigma \bigg )\tau \bigg (\nabla_\bu\bF(\bu_h)\cdot \nabla \bu_h \bigg )\;d\bx
    \end{equation}
or
\begin{equation}
\label{schema RDS jump}\Phi_\sigma^K(\bu_h)=\beta_\sigma \int_{\partial K}\bF(\bu_h)\cdot \bn\; d\gamma+
    \Gamma \;h_K^2 \int_{\partial K} [\nabla \bu_h]\cdot [\nabla \varphi_\sigma]\; d\gamma.\end{equation}
   where the parameters $\beta_\sigma$ are defined to guarantee conservation,
   $$\sum\limits_{\sigma\in K} \beta_\sigma=1$$
   and such that \eqref{schema RDS SUPG} without the streamline term and \eqref{schema RDS jump} without the jump terms satisfy a discrete maximum principle. The streamline term and jump term are introduced because one can easily see that spurious modes may exist, but their role is very different compared to \eqref{SUPG} and \eqref{Jump} where they are introduced to stabilize the Galerkin scheme: if formally the maximum principle is violated, experimentally the violation is extremely small, if nonexistant. See \cite{energie,abgrallLarat} for more details. 
   
   A similar construction can be done starting from a discontinuous Galerkin scheme without non-linear stabilisation such as limiting, has been applied. This has been done in \cite{abgrall:shu} and developped further in \cite{abgrall:dgrds}. 
   
   The non-linear stability is provided by the coefficient $\beta_\sigma$ which is a non-linear function of $\bu_h$.  Possible values of $\beta_\sigma$ are described in the appendix \ref{RDS}.
   
Here we consider a globally continuous approximation: $\bu_h \in \mathcal{V}_h\cap C^0(\Omega)$.

Consider one element $K$. Since there is no ambiguity, the drop, for the residuals,  any reference to $K$ in the following. The total residual is defined by
$$
\Phi(\bu_h)=\int_{\partial K} \bF(\bu_h)\cdot \bn \; d\gamma,$$
and we assume to have monotone residuals $\{\Phi_\sigma^L\}_{\sigma\in K}$. By this we mean 
$$\Phi_\sigma^L(\bu_h)=\sum\limits_{\sigma'\in K} c_{\sigma\sigma'}^L(\bu_\sigma-\bu_{\sigma'})$$
with $c_{\sigma\sigma'}\geq 0$ that also satisfies
$$\sum\limits_{\sigma\in K}\Phi_\sigma^L=\Phi.$$ It can easily be shown  that the condition $c_{\sigma\sigma'}\geq 0$ garanties that the scheme is monotone under a CFL like condition. One example is 
given by the Rusanov residuals:
$$\Phi_\sigma^{Rus}(\bu_h)=-\int_K \nabla \varphi_\sigma \cdot  \bF(\bu_h)\; d\bx+\int_{\partial K} \varphi_\sigma\bbf(\bu_h)\cdot \bn\; d\gamma+\alpha (\bu_\sigma-\bar \bu),
$$
where $\bar \bu$ is the arithmetic average of of the $\bu_\sigma's$ on $K$ and $\alpha$ satisfies:
$$\alpha\geq \#K\; \max_{\sigma, \sigma'\in K} \bigg | \int_K \varphi_\sigma \nabla_\bu\bF(\bu_h)\;  d\bx \bigg |.$$
Here $\#K$ is the number of degrees of freedom in $K$.
Indeed, this residual can be rewritten as 
$$\Phi_\sigma^{Rus}(\bu_h)=\sum_{\sigma'\in K} c_{\sigma\sigma'}(\bu_\sigma-\bu_{\sigma'})$$
with
$$c_{\sigma\sigma'}=\int_K \varphi_\sigma \cdot \nabla_\bu\bF(\bu_h)\; d\bx+\dfrac{\alpha}{\#K}.$$
Under the condition above, $c_{\sigma\sigma'}\geq 0$ and hence we have a maximum principle.

The coefficients $\beta_\sigma$ introduced in the relations \eqref{schema RDS SUPG} and \eqref{schema RDS jump} are defined by:
$$
\beta_\sigma=\dfrac{\max(0,\frac{\Phi_\sigma^{L}}{\Phi})}
{
\sum\limits_{\sigma'\in K} \max(0,\frac{\Phi_{\sigma'}^{L}}{\Phi})}.
$$ and can be shown to be always defined, to  guaranty a local maximum principle for \eqref{schema RDS SUPG} and \eqref{schema RDS jump}, see \cite{abgrallLarat}.

\begin{remark}[About the coefficients $c_{\sigma\sigma'}^L$]
All the examples of monotone residual we are aware of are such that for linear problems, te $c_{\sigma\sigma'}^L$ are independant of $\bu_h$. Then one can show that for any $\sigma\in K$,
$$\sum\limits_{\sigma'\in K}\big ( c_{\sigma\sigma'}^L-c_{\sigma'\sigma}^L\big )=\int_K \nabla\bbf(\bu_h)\cdot\nabla \varphi_\sigma d\bx.$$
This relation implies the conservation relation \eqref{RD:cons:K}.
\end{remark}

\section{RDS as finite volume schemes}
In this section, we show how to interpret  RD schemes as finite volume schemes. This amounts to defining control volumes and flux functions. 
We first have to define what is a flux in this context and to adapt the notion of consistency. 

Let us  consider any common edge or face $\Gamma$ of $K^+$ and $K^-$, two elements. Let $\bn$ be the normal to $\Gamma$, see Figure \ref{fig:flux}. Depending on the context, $\bn$ is a scaled normal or $||\bn||=1$.
\begin{figure}[h]
\begin{center}
\psfrag{K+}{$K^+$}
\psfrag{K-}{$K^-$}
\psfrag{n}{$\bn$}
\includegraphics[width=0.45\textwidth]{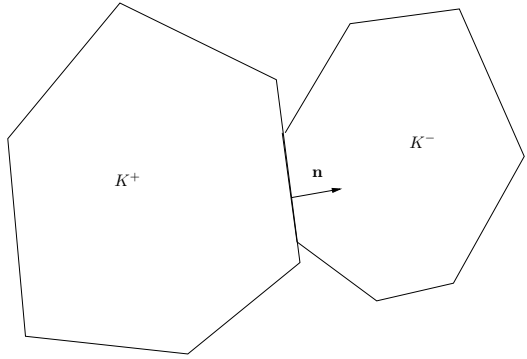}
\caption{\label{fig:flux}Geometrical setting}
\end{center}
\end{figure}
To each edge $\Gamma$ is associated a set of states $S=\{\bbu_1, \ldots , \bbu_l)$. A flux $\hbbf_\bn(S)$ between $K^+$ and $K^-$ has to satisfy
\begin{subequations}\label{flux:def}
\begin{equation}\label{flux:def:1}
\hbbf_\bn(\bbu_1, \ldots , \bbu_l)=-\hbbf_{-\bn}(\bbu_1, \ldots , \bbu_l).
\end{equation}

The consistency property that stands for the consistency is that if all the states  are identical in an element, then each of the residuals vanishes. Hence, we define a multidimensional flux as follows:
 A multidimensional flux 
$$\hbbf_\bn:=\hbbf_\bn(\bu_1, \ldots , \bu_N)$$
is consistent if, when $\bu_1= \bu_2= \ldots = \bu_N=\bu$ then
\begin{equation}\label{flux:def:2}\hbbf_\bn(\bu, \ldots , \bu)=\bbf(\bu)\cdot \bn.\end{equation}
\end{subequations}
The results of this section apply to any 
finite element method but also to discontinuous Galerkin methods. There is no need for exact evaluation of integral formula (surface or boundary), so that these results apply to schemes as they are implemented.

Let  $K$ is a polytope contained in $\R^d$ with degrees of freedoms on the boundary of $K$. The set $\mathcal{S}$ is the set of degrees of freedom in $K$. We consider a triangulation $\mathcal{T}_K$ of $K$ whose vertices  are exactly the elements of $\mathcal{S}$. Choosing an orientation of $K$, it is propagated on $\mathcal{T}_K$: the edges are oriented. This is illustrated in figure \ref{graph} for a $\P^2$ triangle and a $\Q^2$ quad
\begin{figure}[h]
\begin{center}
\subfigure[$\P^2$ triangle]{\includegraphics[width=0.45\textwidth]{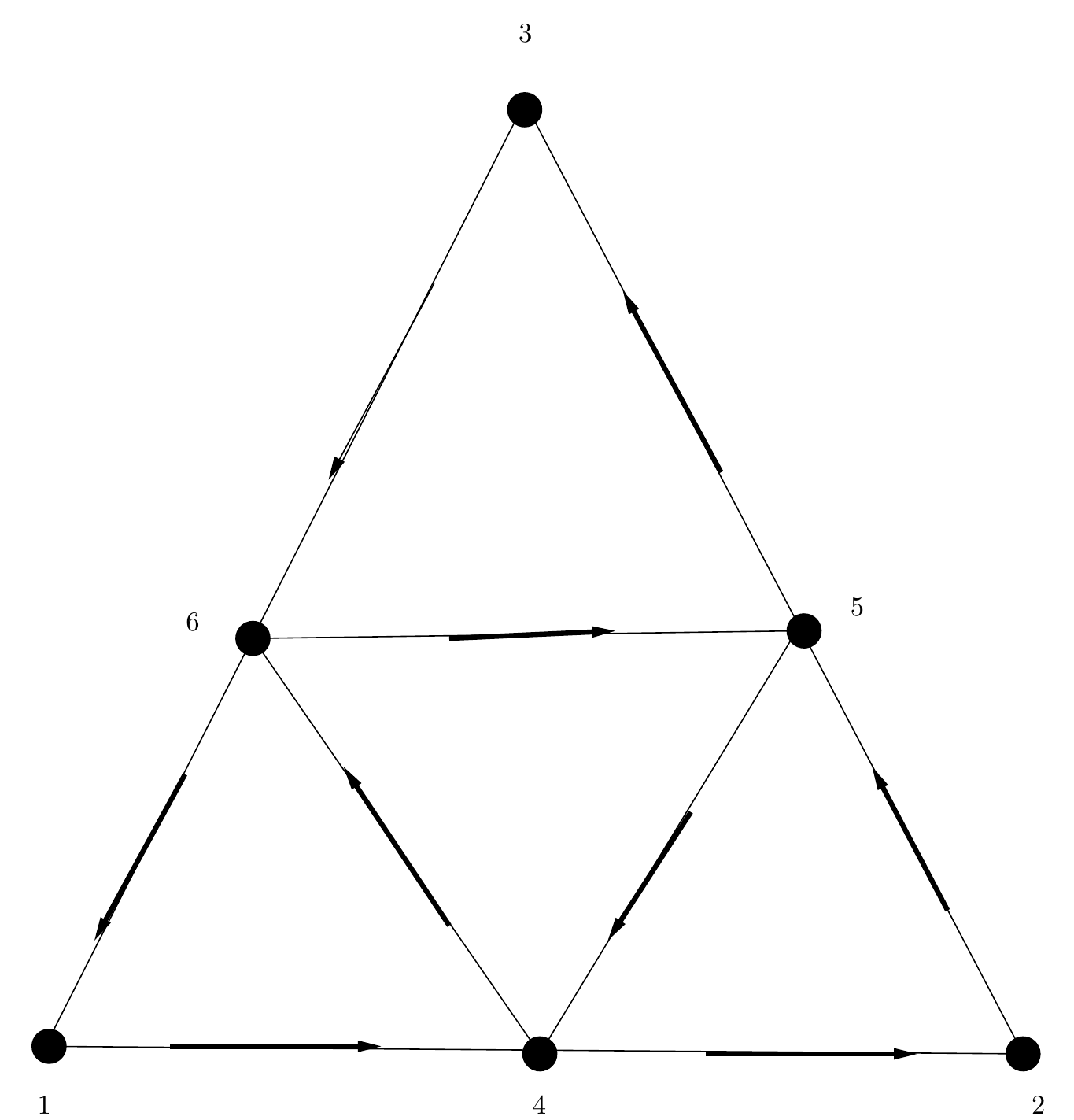}}
\subfigure[$\Q^2$ triangle]{\includegraphics[width=0.45\textwidth]{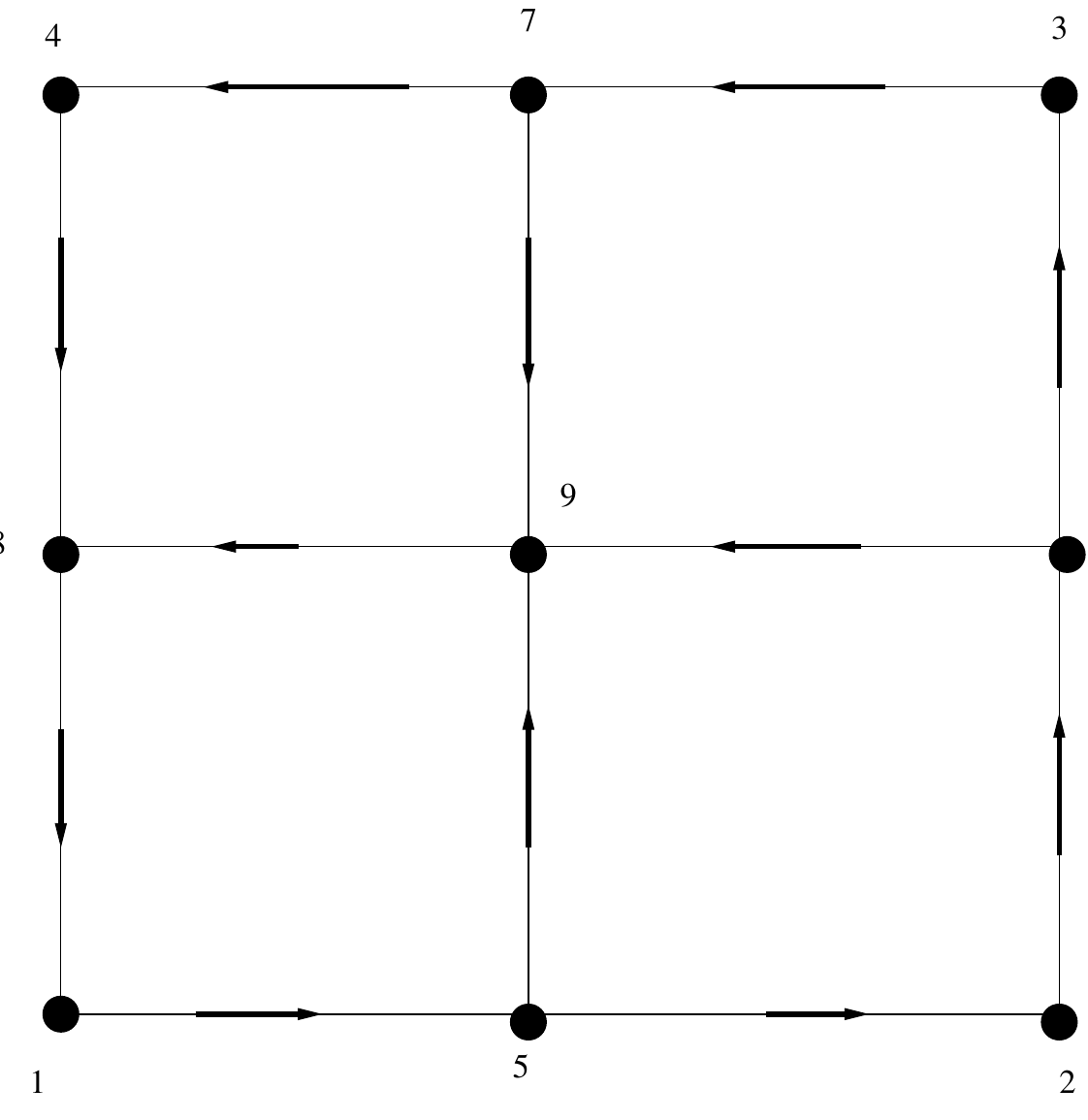}}
\end{center}
\caption{\label{graph} Examples of two oriented graphs.}
\end{figure}

The problem is to find normals $\bn_{\sigma\sigma'}$ with  $\bn_{\sigma\sigma'}=-\bn_{\sigma'\sigma}$ and quantities $\hbbf_{\bn_{\sigma\sigma'}}:=\hbbf_{\sigma,\sigma'}$ for any edge $[\sigma,\sigma']$ of   $\mathcal{T}_K$ such that:
\begin{subequations}\label{GC:1}
\begin{equation}
\label{GC:1.1}
\Phi_\sigma=\sum_{\text{ edges }[\sigma,\sigma']} \hbbf_{\sigma,\sigma'}+\hbbf_\sigma^{b}
\end{equation}
with
\begin{equation}
\hbbf_{\sigma,\sigma'}=-\hbbf_{\sigma',\sigma}
\label{GC:1.2}
\end{equation}
and $\hbbf_\sigma^{b}:=\hbbf_\sigma^{b}(\bbu_1, \ldots, \bbu_{k_\sigma})$ is the 'part' of $\oint_{\partial K} \hbbf_\bn(\bu_h,\bu^{h,-}) \; d\gamma$ associated to $\sigma$. {The control volumes will be defined by their normals so that we get consistency.}

Note that \eqref{GC:1.2} implies the conservation relation
\begin{equation}
\label{GC:conservation}
\sum\limits_{\sigma\in K}\Phi_\sigma=\sum\limits_{\sigma\in K}\hbbf_\sigma^b.
\end{equation}
In short, we will consider 
\begin{equation}
\label{BC:1.3}
\hbbf_\sigma^b=\oint_{\partial K} \varphi_\sigma\; \hbbf_\bn (\bu_h,\bu^{h,-}) \; d\gamma,
\end{equation}
\end{subequations}
but other  examples can be considered provided the consistency \eqref{GC:conservation} relation holds true: the choice is a bit arbitrary, provided \eqref{GC:conservation} holds true.

Any edge $[\sigma,\sigma']$ is either direct or, if not, $[\sigma',\sigma]$ is direct. Because of \eqref{GC:1.2}, we only need to know $\hbbf_{\sigma,\sigma'}$ for direct edges. Thus we introduce the notation $\hbbf_{\{\sigma,\sigma'\}}$ for  the flux  assigned to  the direct edge whose extremities are $\sigma$ and $\sigma'$. We can rewrite \eqref{GC:1.1} as, for any $\sigma\in \mathcal{S}$,
\begin{equation}
\label{GC:1.1bis}
\sum_{\sigma'\in \mathcal{S}} \varepsilon_{\sigma,\sigma'} \hbbf_{\{\sigma,\sigma'\}}=\Psi_\sigma:=\Phi_\sigma-\hbbf_\sigma^b,
\end{equation}
with $$
\varepsilon_{\sigma,\sigma'}=\left \{
\begin{array}{ll}
0& \text{ if }\sigma \text{ and }\sigma' \text{ are not on the same edge of }\mathcal{T},\\
1& \text{ if } [\sigma,\sigma']\text{ is an edge and } \sigma \rightarrow \sigma' \text{ is direct,}\\
-1&  \text{ if } [\sigma,\sigma']\text{ is an edge and } \sigma' \rightarrow \sigma \text{ is direct.}
\end{array}
\right .
$$
$\mathcal{E}^+$ represents the set of direct edges.

Hence the problem is to find  a vector $\hbbf=(\hbbf_{\{\sigma,\sigma'\}})_{\{\sigma,\sigma'\} \text{ direct edges}}$ such that
$$A\hbbf=\Psi$$
where $\Psi=(\Psi_\sigma)_{\sigma\in \mathcal{S}}$ and $A_{\sigma \sigma'}=\varepsilon_{\sigma,\sigma'}$.

We have  the following lemma which shows the existence of a solution.
\begin{lemma}\label{lemma:flux}
For any couple $\{\Phi_\sigma\}_{\sigma\in \mathcal{S}}$ and $\{\hbbf_\sigma^{b}\}_{\sigma\in \mathcal{S}}$ satisfying the condition  \eqref{GC:conservation}, there exists numerical flux functions $\hbbf_{\sigma,\sigma'}$ that satisfy \eqref{GC:1}. Recalling that the  matrix of the Laplacian of the graph is $L=AA^T$, we have
\begin{enumerate}
\item The rank of $L$ is $|\mathcal{S}|-1$ and its image is $\big (\text{span}\{\mathbf{1}\})^\bot$. We still denote the inverse of $L$ on $\big (\text{span}\{\mathbf{1}\} )^\bot$ by $L^{-1}$,
\item 
With the previous notations, a solution is 
\begin{equation}
\label{eq:lemma}\big (\hbbf_{\{\sigma,\sigma'\}}\big )_{\{\sigma,\sigma'\} \text{ direct edges}}=A^TL^{-1} \big (\Psi_\sigma\big )_{\sigma\in \mathcal{S}}.\end{equation}
\end{enumerate}
\end{lemma}
The proof can be found in \cite{AbgrallConservation}. The computation of $A^TL^{-1}$ is easy in practice: since $L$ is symetric, the range of $L$ is orthogonal to its kernel, spanned by $x_0=(1,\ldots , 1)^T$. Hence, for any $\lambda\neq 0$, the matrix $L+\lambda \frac{x_0\otimes x_0}{||x_0||^2}$ is invertible, and the matrix written as $L^{-1}$ with some abuse of language is $(L+\lambda \frac{x_0\otimes x_0}{||x_0||^2})^{-1}- \frac{x_0\otimes x_0}{\lambda||x_0||^2}$: the computation can be done by any standard matrix inversion package.

If in addition, the boundary flux satisfy for any $\bu$
\begin{equation}\label{consistency:bFlux}
\hbbf_\sigma^b(\bbu, \ldots , \bbu)=\bbf(u)\cdot \mathbf{N}_\sigma,
\end{equation}
as for example \eqref{BC:1.3}, then this set of flux are consistent and the  normals $\bn_{\sigma,\sigma'}$ are given by
\begin{equation}
\label{GC:normals}
\big ( \bn_{\sigma,\sigma'}\big )_{[\sigma,\sigma']\in \mathcal{E}^+}=A^TL^{-1} \big ( \mathbf{N}_{\sigma_1}, \ldots , \mathbf{N}_{\sigma_{\#K}}\big )^T
\end{equation}
 We can state:
\begin{proposition}
If the residuals $(\Phi_\sigma)_{\sigma\in K}$ and the boundary fluxes $(\hbbf_\sigma^b)_{\sigma\in K}$ satisfy \eqref{GC:conservation} and \eqref{consistency:bFlux}, then we can find  a set of consistent flux $(\hbbf_{\sigma,\sigma'})_{[\sigma,\sigma']} $ satisfying \eqref{GC:1}. They are given by \eqref{eq:lemma}. In addition, for a constant state,
$$\hbbf_{\sigma,\sigma'}(\bu_h)=\bbf(\bu_h)\cdot\bn_{\sigma,\sigma'}$$ for the normals defined by \eqref{GC:normals}.
\end{proposition}
This also defines the control volumes since we know their normals.

\bigskip

We can state a couple of general remarks:
\begin{remark}
\begin{enumerate}
\item The flux $\hbbf_{\sigma,\sigma'}$ depend on the $\Psi_\sigma$ and not directly on the $\hbbf_\sigma^b$. We can design the fluxes independently of the boundary flux, and their consistency directly comes from the consistency of the boundary fluxes.
\item 
The residuals depends on more than 2 arguments. For stabilized finite element methods, or the non linear stable residual distribution
 schemes, see e.g.  \cite{Hughes1,struijs,abgrallLarat}, the residuals depend on all the states on  $K$. Thus
the formula \eqref{eq:lemma} shows that the flux depends on more than two states in contrast to the  1D case. In the finite volume case however, the support of the flux function is generally larger than the three states of $K$, think for example of an ENO/WENO method, or a simpler MUSCL one.
\item The formula \eqref{eq:lemma} are influenced by the form of the total residual \eqref{fv:tot:residu}.  We show in the next paragraph how this can be generalized.
\item The formula \eqref{eq:lemma} make no assumption on the approximation space $v_h$: they are valid for continuous and discontinuous approximations. The structure of the approximation space appears only in the total residual.
\item Quadrature formula: all the relations we use are obtained by quadrature formula. This means that the integration does not need to be exact.
\end{enumerate}

\end{remark}

To end this paragraph, let us give one example. We consider a triangle and a quadratic approximation, see figure \ref{figP2}
\begin{figure}[h]
\begin{center}
\includegraphics[width=0.45\textwidth]{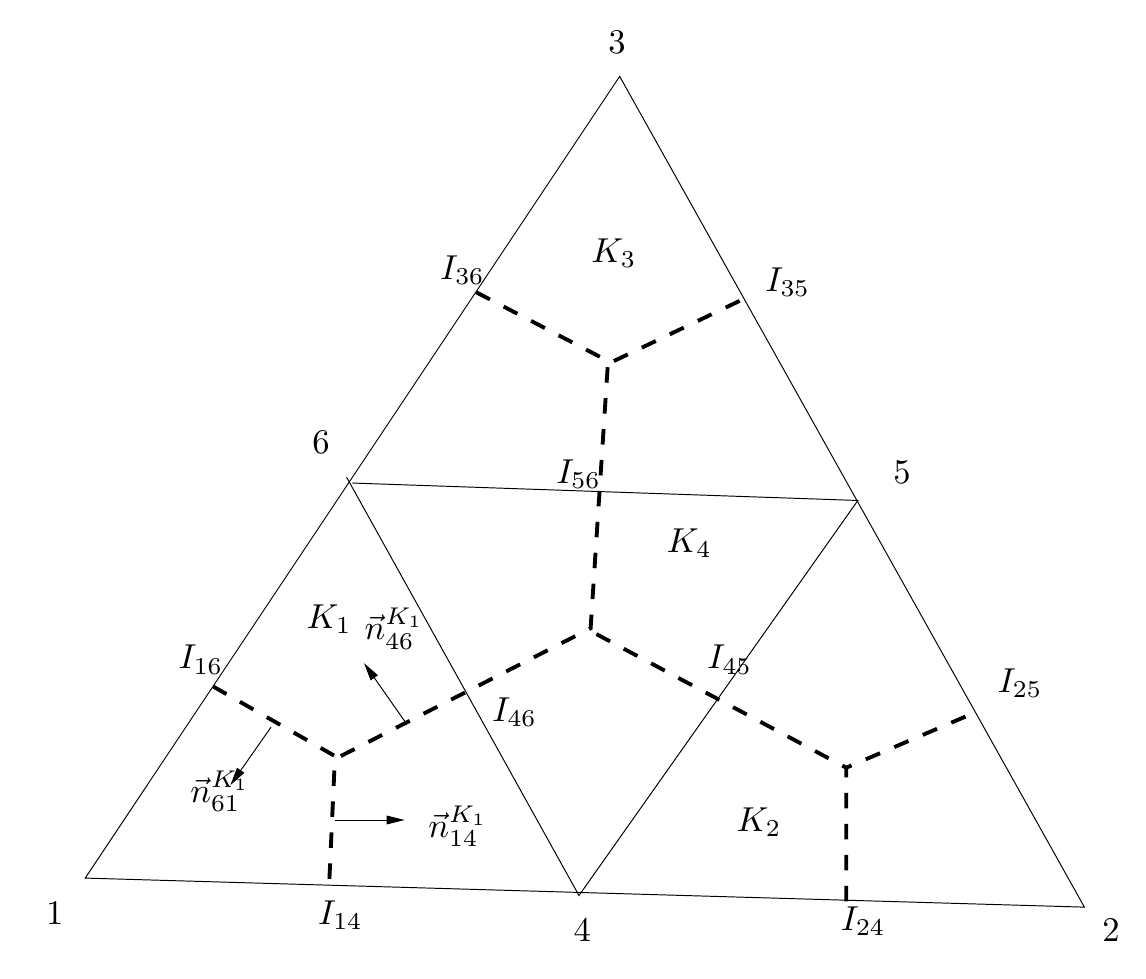}
\end{center}
\caption{\label{figP2} Example of a quadratic triangle. The degrees of freedom are $\sigma=1, \ldots, 6$. The graph is shown with plain lines. Doted lines: the control volumes  (for the discontinuous case) and   the intersection of the control volumes with the elements for a globally continuous approximation. The normal are also represented.
}
\end{figure}
The one can easily show  that
$$
\begin{array}{lcl}
\hbbf_{14}&=&\dfrac{1}{12}\big (\Psi_1-\Psi_4\big )+\dfrac{1}{36}\big ( \Psi_6-\Psi_5\big )+\dfrac {7}{36}\big (\Psi_1- \Psi_2\big )+\dfrac {5 }{36}\big (\Psi_3-\Psi_1\big )\\
&\\
\hbbf_{16}&=&\dfrac{1}{12}\big ( \Psi_4-\Psi_1\big )+\dfrac {5}{36}\big ( \Psi_5-\Psi_1)
+\dfrac {7}{36}\big ( \Psi_6-\Psi_1\big ) +\dfrac{1}{36}\big ( \Psi_3- \Psi_2\big )\\
&\\
\hbbf_{46}&=&\dfrac{2}{9}\big (\Psi_2-\Psi_6\big )+\dfrac{1}{9}\big (  \Psi_3- \Psi_5\big )\\
&\\
\hbbf_{54}&=&\dfrac{2}{9}\big (  \Psi_5-\Psi_2\big )+\dfrac{1}{9}\big ( \Psi_5-\Psi_1\big )\\
\end{array}
$$
$$
\begin{array}{lcl}
\hbbf_{42}&=&\dfrac {7}{36}\big (\Psi_2-\Psi_3\big ) +\dfrac{5}{36}\big (\Psi_1-\Psi_3\big )+\dfrac{1}{12}\big(\Psi_6-\Psi_3\big )+\dfrac{1}{36}\big (\Psi_5-\Psi_4\big ) \\

&\\
\hbbf_{25}&=&\dfrac{1}{36}\big (\Psi_2-\Psi_1\big )+\dfrac{5}{36}\big (\Psi_3-\Psi_5\big ) +\dfrac{7}{36}\big ( \Psi_3-\Psi_5\big )+\dfrac{1}{12}\big (\Psi_3-\Psi_6\big )  \\

&\\
\hbbf_{53}&=&\dfrac{1}{36}\big (\Psi_1-\Psi_6\big )+\dfrac{5}{36}\big (\Psi_3-\Psi_5\big )+\dfrac{7}{36}\big (\Psi_4-\Psi_5\big )+\dfrac{1}{12}\big (\Psi_2-\Psi_5\big )
\\
&\\
\hbbf_{63}&=& \dfrac{1}{36}\big (\Psi_4-\Psi_3\big )+\dfrac{5}{36}\big (\Psi_5-\Psi_1\big )+\dfrac{7}{36}\big (\Psi_5-\Psi_6\big )+\dfrac{1}{12}\big (\Psi_5-\Psi_2\big )\\
&\\

\hbbf_{65}&=&\dfrac{1}{9}\big (\Psi_1- \Psi_3\big )+\dfrac{2}{9}\big ( \Psi_6- \Psi_4\big )\end {array}
$$
Then we choose the boundary flux:
$$\hbbf_\sigma^b=\int_{\partial K}\hf_\bn \varphi_\sigma\; d\gamma$$ and get:
$$
\begin{array}{lll}
\normal_l=-\dfrac{\bn_l}{6} & \text{if }l=1,2,3\\ &&\\
\normal_4=\dfrac{\bn_3}{3}& \normal_5=\dfrac{\bn_1}{3}& \normal_6=\dfrac{\bn_2}{3}
\end{array}
$$
The normals are given by:
$$
\begin{array}{lcl}
\bn_{14}&=&\dfrac{1}{12}\big (\normal_1-\normal_4\big )+\dfrac{1}{36}\big ( \normal_6-\normal_5\big )+\dfrac {7}{36}\big (\normal_1- \normal_2\big )+\dfrac {5 }{36}\big (\normal_3-\normal_1\big )\\
&\\
\bn_{16}&=&\dfrac{1}{12}\big ( \normal_4-\normal_1\big )+\dfrac {5}{36}\big ( \normal_5-\normal_1)
+\dfrac {7}{36}\big ( \normal_6-\normal_1\big ) +\dfrac{1}{36}\big ( \normal_3- \normal_2\big )\\
&\\

\bn_{46}&=&\dfrac{2}{9}\big (\normal_2-\normal_6\big )+\dfrac{1}{9}\big (  \normal_3- \normal_5\big )\\
&\\
\bn_{54}&=&\dfrac{2}{9}\big (  \normal_5-\normal_2\big )+\dfrac{1}{9}\big ( \normal_5-\normal_1\big )
\end{array}
$$
$$
\begin{array}{lcl}
\bn_{42}&=&\dfrac {7}{36}\big (\normal_2-\normal_3\big ) +\dfrac{5}{36}\big (\normal_1-\normal_3\big )+\dfrac{1}{12}\big(\normal_6-\normal_3\big )+\dfrac{1}{36}\big (\normal_5-\normal_4\big ) \\

&\\
\bn_{25}&=&\dfrac{1}{36}\big (\normal_2-\normal_1\big )+\dfrac{5}{36}\big (\normal_3-\normal_5\big ) +\dfrac{7}{36}\big ( \normal_3-\normal_5\big )+\dfrac{1}{12}\big (\normal_3-\normal_6\big )  \\

&\\
\bn_{53}&=&\dfrac{1}{36}\big (\normal_1-\normal_6\big )+\dfrac{5}{36}\big (\normal_3-\normal_5\big )+\dfrac{7}{36}\big (\normal_4-\normal_5\big )+\dfrac{1}{12}\big (\normal_2-\normal_5\big )
\\
&\\
\bn_{63}&=& \dfrac{1}{36}\big (\normal_4-\normal_3\big )+\dfrac{5}{36}\big (\normal_5-\normal_1\big )+\dfrac{7}{36}\big (\normal_5-\normal_6\big )+\dfrac{1}{12}\big (\normal_5-\normal_2\big )\\
&\\

\bn_{65}&=&\dfrac{1}{9}\big (\normal_1- \normal_3\big )+\dfrac{2}{9}\big ( \normal_6- \normal_4\big )\end {array}
$$

\section{Satisfaction of constraints}
In this section, we want to show that  this notion of conservation at the level of elements can also be useful to construct, from a known scheme, a new one that satisfy additional constraints. More explicitly, let us consider the two different problems:
\begin{itemize}
\item Assume there is a function $\bbu \in \R^m\mapsto \bU(\bbu)\in \R^{m\times m}$ such that the (possibly)  non conservative non linear PDE
\begin{subequations}\label{ExampleA}
\begin{equation}\label{ExampleA:1}\dpar{\bbu}{t}+\mathbf{A}\cdot \nabla \bbu=0\end{equation}
can be put in conservation form by:
\begin{equation}
\label{ExampleA:2}\bU(\bbu)^T \dpar{\bbu}{t}=\dpar{\bw}{t}, \qquad \bU(\bbu)^T \mathbf{A}\cdot \nabla \bbu=\text{ div }\bbf(\bw).\end{equation}
\end{subequations}
An example, taken from fluid mechanics,  is 
$$\bbu=\begin{pmatrix}\rho \\ \rho \bv \\ e\end{pmatrix}, \mathbf{A}\cdot \nabla \bbu=\begin{pmatrix}
\text{ div } (\rho \bv )\\\text{ div }\big (\rho \bv\otimes \bv +p\text{Id}\big )\\
\bv\nabla e+ (e+p)\text{ div }\bv\end{pmatrix}.
$$
If we take $$\bU=\begin{pmatrix} 1& \mathbf{0} &0 \\
0& \mathbf{Id}& 0\\
-\frac{\bv^2}{2} & \bv & 1
\end{pmatrix}
$$
then we recover the conservative form of the Euler equations.
\item Assume formally that the conservative system
$$\dpar{\bbu}{t}+\text{ div }\bbf(\bbu)=0$$ satisfies an additional conservation relation,
\begin{equation}\label{ExampleB}\dpar{E}{t}+\text{ div} \bbg(\bbu)=0\end{equation}
where $E$ is a function of the state $\bbu$. Assume in addition, that we satisfy the second system  by some algebraic manipulations, for example there exists a mapping $\bbu\in \R^m \mapsto \bv(\bbu)\in \R^m$ such that
$$
\bv^T\dpar{\bbu}{t}=\dpar{E}{t}, \qquad \bv^T\text{ div }\bbf(\bbu)=\text{ div }\bbg(\bu).$$
In other words,
$$\bv=\nabla_\bbu E, \bv^T\nabla_\bu\bbf=\nabla_\bbu \bbg.$$
An example is the entropy $E$.
\end{itemize}

We first show how, starting from a known scheme, we can modify it so that the new scheme will satisfy the additional constraints.  Since both problems are unsteady, we first explain how we discretise them in space and time, then we show how to modify the schemes.
\subsection{Discretisation in space and time}
Let us consider \eqref{eq:1} and follow \cite{Mario,Abgrall2017}.  If one wants to discretize this problem starting from the schemes of the type \eqref{RDS} where the residuals are given by \eqref{SUPG}, \eqref{Jump} or  \eqref{schema RDS}, and if we want to keep the automatic consistency with the original PDE that is provided by the residual formulation,  we are led to method where there is a mass matrix.  This is also true with \eqref{DG:var}, but here the mass matrix is block diagonal  (so easy to invert)contrarily to the other cases where it is only sparse. In addition, in the case \eqref{schema RDS}, it is very unclear that this mass matrix (which is formal in that case) is \ldots invertible: the trick introduced in \cite{Mario} is precisely done to avoid to invert the mass matrix, to the price of adding some dissipative term. In \cite{Abgrall2017}, this trick was reinterpreted as a Defect Correction approach and then generalised to any order. 
In detail the formal scheme will write as: find $\bu_h\in v_h$ such that for any test function $v_h$m, we have
\begin{equation}\label{unsteady}\langle \dpar{\bu_h}{t},v_h \rangle + a(\bu_h, v_h)=0\end{equation}
where $a$ is the form defined in \eqref{SUPG}, \eqref{Jump} or  \eqref{schema RDS} that we can write as
$$a(\bu_h, v_h)=\sum_{K\subset \Omega} \int_K v_h \; \text{ div }\bbf(\bbu_h) \; d\bx +J(\bu_h,v_h)$$,
$J$ is a possible jump term (such as in \eqref{Jump} and \eqref{schema RDS jump} and $v_h$ a test function that is
\begin{itemize}
\item constant by element $K$ for \eqref{schema RDS SUPG}  and \eqref{schema RDS jump},
\item any element of $v_h$ for \eqref{Jump},
\item of the form $w_h+h_K\tau_K \nabla_u\bbf(\bu_h)\cdot \nabla w_h$ for \eqref{SUPG}
\end{itemize}
Then 
$$\langle \dpar{\bu_h}{t},v_h \rangle=\sum_{K\subset \Omega} \int_K v_h \dpar{\bu_h}{t}\; d\bx.$$ 
A priori, the idea is to start from
\begin{equation}\label{star}\langle \bu_h(t),v_h\rangle=\langle \bu_h(t_n),v_h\rangle-\int_{t_n}^{t} a(\bu_h(s), v_h) \; ds\end{equation}
and to use some quadrature formula. Note that the test function do not depend on time. This amount to subdivide the interval $[t_n,t_{n+1}]$ with sub-timesteps $t_{n,0}=t_n<t_{n,1}<\ldots< t_{n,p-1}<t_{n,p}=t_{n+1}$ and  to approximate the relation \eqref{star} at the sub-timesteps. If $\bu_h$ is the vector $(\bu_h^n, \bu_h^{n,1}, \ldots, \bu_h^{n,p-1}, \bu_h^{n+1})$ with $\bu_h^{n,l}\approx\bu(t_{n,l})$, we write the approximation as
\begin{equation}\label{starstar}
\langle \bu_h,v_h\rangle=\langle \bu_h(t_n),v_h\rangle- \Delta t A(\bu_h, v_h) \end{equation}
For example,  the Crank-Nicholson method leads to
$$\begin{pmatrix}
\langle \bu_h^{n+1},v_h\rangle\\\langle \bu_h^{n},v_h\rangle\end{pmatrix}=\begin{pmatrix}\langle \bu_h^{n},v_h\rangle\\\langle \bu_h^{n},v_h\rangle\end{pmatrix}-\Delta t
\begin{pmatrix}\dfrac{a(\bu_h^{n+1}, v_h)+a(\bu_h^n, v_h)}{2}\\0\end{pmatrix}
$$
and $$A(\bu_h,v_h):=\begin{pmatrix}\dfrac{a(\bu_h^{n+1}, v_h)+a(\bu_h^n, v_h)}{2}\\0\end{pmatrix}.$$
In \cite{Abgrall2017}, knowing $\bu_h^n$, we compute $\bu_h^{n+1}$ by the following algorithm (we just show the simplest version).
\begin{itemize}
\item Set $\bu^{(0)}=\bu_h^n$,
\item for $p=1,\ldots , N$,  compute $\bu^{(p+1)}$ by solving
$$\langle\langle \bu^{(p+1)}, v_h\rangle\rangle =\langle\langle \bu^{(p)}, v_h\rangle\rangle- \langle \bu^{(p)}-\bu^{n}, v_h\rangle -A(\bu^{(p)}, v_h)$$
where
$$\langle\langle \bu_h,\bv_h\rangle\rangle=\sum_{K\subset \Omega} C_K \sum_{\sigma\in K}(\bu_h)_\sigma( v_h)\sigma.$$
The forms $a$ are written in term of spatial residuals of the type \eqref{SUPG}, \eqref{Jump} or \eqref{schema RDS SUPG}  and \eqref{schema RDS jump}.
\end{itemize}
In practice $N$ is equal to the number of sub-time steps, and the accuracy is not spoiled if some technical condition described in \cite{Abgrall2017} are met. They are met in particular is the degrees of freedoms $\sigma$ are the control points of the B\'ezier polynomials of degree $N$. 

\subsection{Corrected scheme: the example \eqref{ExampleA}}
We have a  set of residuals $\{\Phi_\sigma^K\}_{\sigma\in K}$ that satisfy the 'conservation' relations:
$$\sum_{\sigma\in K}
\Phi_\sigma^K=\oint_K\mathbf{A}\cdot \nabla \bbu_h \; d\bx.$$
The surface integral is approximated by a quadrature formula such that the error is $|K|O(h^{k+1})$, where $k$ is the polynomial degree.

If there exists some average of $\bU(\bbu)$, say $\overline{\bU}(\bbu^{(p+1)},\bbu^{(p)})$ such that
$$
\overline{\bU}(\bbu^{(p+1)},\bbu^{(p)})^T\; \big ( \bbu_h^{(p+1)}-\bbu_h^{(p)}\big ) =\bw^{(p+1)}-\bw^{(p)},$$
then the conservation constraints are recovered if for any $K$, we have:
\begin{equation}
\begin{split}
\sum_{\sigma \in K} \overline{\bU}(\bbu^{(p+1)},\bbu^{(p)})_\sigma^T\Phi_\sigma^K(\bu^n, \bu^{n,1,(p)}, &\ldots , \bu^{n+1,(p)} ) =\oint_K\big ( \bw^{(p)}-\bw^{(n)}\big )\; d\bx\\
&\qquad +\oint_{\partial K}\mathcal{I}(\bbf(\bbu^{n})\cdot\bn, \bbf(\bbu^{n,1,(p)})\cdot\bn, \ldots , \bbf(\bbu^{n+1(p)})\cdot\bn) \; d\gamma\\
\sum_{\sigma \in \Gamma} \overline{\bU}(\bbu^{(p+1)},\bbu^{(p)})_\sigma^T\Psi_\sigma^\Gamma(\bu^n, \bu^{n,1,(p)}, &\ldots , \bu^{n+1,(p)} )=\oint_{\Gamma}\mathcal{I}
(\bbf(\bbu^{n})\cdot\bn, \bbf(\bbu^{n,1,(p)})\cdot\bn, \ldots , \bbf(\bbu^{n+1,(p)})\cdot\bn)
\; d\gamma
\end{split}
\end{equation}
where $\Phi_\sigma^K(\bu^n, \bu^{n,1,(p)}, \ldots , \bu^{n+1,(p)} ) $ are the residual that are needed to define $A$ in \eqref{starstar} and  $\mathcal{I}
(\bbf(\bbu^{n})\cdot\bn, \bbf(\bbu^{n,1,(p)})\cdot\bn, \ldots , \bbf(\bbu^{n+1,(p)})\cdot\bn)$ is a weighted average of the normal flux at the sub-timesteps after the $p$-th iteration, this translate the way \eqref{starstar} discretise \eqref{star}. In the case of the Crank-Nicholson method, this is simply an arithmetic average between the states at $t_n$ and $t_{n+1}$.
Of course these relations are in general not satisfied by the original residuals, and we modify them by adding a correction
$$\Phi_\sigma^K\rightarrow \Phi_\sigma^K+r_\sigma^K, \qquad \Psi_\sigma^\Gamma\rightarrow \Psi_\sigma^\Gamma+r_\sigma^\Gamma.$$
We only describe what happens for the elements $K$.

The correction $r_\sigma^K$ must satisfy
\begin{equation}\label{average}
\begin{split}
\sum_{\sigma\in K} \overline{\bU}(\bbu^{(p+1)},\bbu^{(p)})_\sigma^Tr_\sigma^K&=\oint_K\big ( \bw^{(p)}-\bw^{(n)}\big )\; d\bx+\oint_{\partial K}\mathcal{I}(\bbf(\bbu^{n})\cdot\bn, \bbf(\bbu^{n,1,(p)})\cdot\bn, \ldots , \bbf(\bbu^{n+1,(p)})\cdot\bn) \; d\gamma\\
&\qquad -\sum_{\sigma \in K} \overline{\bU}(\bbu^{(p+1)},\bbu^{(p)})_\sigma^T\Phi_\sigma^K(\bu^n, \bu^{n,1,(p)}, \ldots , \bu^{n+1,(p)} ) .
\end{split}
\end{equation}
We have one (vectorial) relations, and at least 2 unknowns (since an element has at least two degrees of freedom). So by a simple linear algebra argument, there might be a solution.
The problem is to find the average in such a way that the explicit nature of the scheme is kept, and second to compute the correction: this is a case by case situation. We show with the example of fluid mechanics how this can be achieved and  show it on a specific example: second order of accuracy.

Using a piecewise linear interpolation in time, the conservation relations are, for $p=0,1$, 
$$\sum_{\sigma \in K} \overline{\bU}(\bbu^{(p+1)},\bbu^{(p)})_\sigma^T\Phi_\sigma^K(\bu^n, \bu^{(p)} ) =\oint_K\big ( \bw^{(p)}-\bw^{(n)}\big )\; d\bx+\oint_{\partial K}\dfrac{\bbf(\bbu^{n}\cdot\bn+ \bbf(\bbu^{(p)})\cdot\bn}{2}\; d\gamma.$$
Here we haver written $\bu^{(p)}$ for $\bu^{n+1,(p)}$ since there is no ambiguity.
Let us write explicitly the differences $\bw^{(p)}-\bw^{(n)}$ in term of the primitive variable.  In the sequel, for any quantity $f$, $\Delta f=f^{(p+1)}-f^{(p)}$.  

Since the density is a conservative and a primitive variable, there is nothing to write.
\begin{equation*}
\begin{split}
\Delta (\rho \bu)&= \rho^{(p+1)} \Delta \bu + \bu^{(p)} \Delta \rho\\
\Delta E&=\Delta e+\frac{1}{2}\Delta  ( \rho\bu^2)\\
& =\Delta e+\dfrac{\bu^{(p+1)}+\bu^{(p)}}{2} \Delta (\rho \bu)-\dfrac{\bu^{(p+1)}\cdot \bu^{(p)}}{2}\Delta \rho
\end{split}
\end{equation*}
i.e. in matrix form:
\begin{equation*}
\begin{pmatrix}
\Delta \rho\\
\Delta (\rho\bu)\\
\Delta E
\end{pmatrix}=\begin{pmatrix}
1 & \mathbf{0} & 0\\
\bu^{(p)}& \rho^{(p+1)} \Id& 0 \\
-\dfrac{\bu^{(p+1)}\bu^{(p)}}{2}& \dfrac{\bu^{(p)}+\bu^{(p+1)}}{2} & 1
\end{pmatrix} \begin{pmatrix}
\Delta \rho\\ \Delta \bu\\ \Delta e \end{pmatrix}
\end{equation*}
The matrix is lower triangular, so that the scheme can be kept explicit: we first compute the density: we know $\rho^{(p+1)}, \rho^{(p)}, \bu^{(p)}, e^{(p)}$. Then we compute the velocity $\bu^{(p+1)}$: we know $\rho^{(p)}, \rho^{(p+1)}, \bu^{(p)}, \bu^{(p+1)}, e^{(p)}$, and then the internal energy. The last question is how to evaluate the corrections. There is no correction on the  density component of the residuals. For the velocity component we write
$$
\sum_{\sigma \in K} \rho^{(p+1)}_\sigma (r_\bu^K)_\sigma=\oint_K \big ( \rho^{(p)}\bu^{(p)}-\rho^{n}\bu^{n}\big )\; d\bx +\oint_{\partial K} \dfrac{\rho^{(p)}\bu^{(p)}+\rho^{n}\bu^{n}}{2}\cdot \bn\; d\gamma -\sum_{\sigma\in K} \bu^{(p)}_\sigma (\Phi_\rho^K)_\sigma.$$
In this relation, the right hand side can be explicitly computed from what is  known: we have one vectorial relation and as many unknown as degrees of freedom in $K$. A priori, there is no reason why to weight differently the degrees of freedom, so we assume
$$(r_\bu^K)_\sigma=r_\bu^K,$$ and then
$$r_\bu^K=\bigg ( \oint_K \big ( \rho^{(p)}\bu^{(p)}-\rho^{n}\bu^{n}\big )\; d\bx +\oint_{\partial K} \dfrac{\rho^{(p)}\bu^{(p)}+\rho^{n}\bu^{n}}{2}\cdot \bn\; d\gamma -\sum_{\sigma\in K} \bu^{p}_\sigma (\Phi_\rho^K)_\sigma \bigg )\bigg (\sum_{\sigma \in K} \rho^{(p+1)}_\sigma\bigg )^{-1}.$$
The same method is used for the energy correction, and we get
\begin{equation*}
\begin{split}
r_e^K&=\frac{1}{\# K}\bigg ( \int_K \big ( E^{(p)}-E^{n}\big ) \; d\bx + \int_{\partial K} \frac{1}{2}\bigg ( \bu^{(p)}\cdot\bn (E^{p})+p^{(p)})+\bu^{n}\cdot\bn (E^n +p^{(0)})\bigg )\; d\Gamma\\ &\qquad -\sum_{\sigma\in K}\bigg [ \dfrac{\bu^{(p)}_\sigma+\bu^{(p+1)}_\sigma}{2}\cdot (\Phi_{\rho \bu}^K)_\sigma
 +\frac{1}{2}\dfrac{\bu^{(p+1)}\bu^{(p)}}{2}(\Phi_\rho^K)_\sigma \bigg ]\bigg )
\end{split}
\end{equation*}
where the $\Phi_{\rho \bu}^K$ are corrected residuals and $N_K$ the number of degrees of freedom in the element.
\subsection{Corrected scheme: the example \ref{ExampleB}}
In that case, the same kind of trick is used, with a small difference: the corrections $r_\sigma$ must not destroy the initial conservation law, and we must have  a relation of the type \eqref{average}. In order to illustrate, we give the  example of fluid mechanics with the mathematical entropy $S=p\rho^{-\gamma+1}$. In the sequel, we denote by $\kappa=\gamma-1$ and recall that $p=\kappa e$.

We can write that $$\Delta S=(\rho^{(p+1)})^{-\kappa}\Delta p+p^{(p+1)} \Delta \rho^{-\kappa}=(\rho^{(p+1)})^{-\kappa}\Delta p+p^{(p+1)}\widetilde{\rho^{-\kappa}} \Delta \rho$$
where
$$\widetilde{\rho^{-\kappa}}=\left \{ \begin{array}{ll}
-\kappa \big (\rho^{(p)}\big )^{-\gamma} & \text{ if } \rho^{(p+1)}=\rho^{(p)}\\
\dfrac{\Delta (\rho^{-\kappa})}{\Delta \rho} & \text{ else.}
\end{array}\right .$$
Assuming we have a initial scheme for the variables $(\rho, \bu, p)$, combining the techniques of the previous example and the formula bellow, we see that the
residual on the velocity should corrected as previously, while the residuals on the pressure should be corrected such that
\begin{equation*}
\begin{split}
\sum_{\sigma\in K}\frac{1}{\kappa}\big ( (\Phi_p)^K_\sigma+(r_p^K)_\sigma\big ) &+\sum_{\sigma\in K}\bigg [ \dfrac{\bu^{(p)}_\sigma+\bu^{(p+1)}_\sigma}{2}\cdot (\Phi_{\rho \bu}^K)_\sigma 
 +\frac{1}{2}\dfrac{\bu^{(p+1)}\bu^{(p)}}{2}(\Phi_\rho^K)_\sigma \bigg ]\bigg )\\&=\int_K \big ( E^{(p)}-E^{(0)}\big ) \; d\bx + \int_{\partial K} \frac{1}{2}\bigg ( \bu^{(p)}\cdot\bn (E^{p})+p^{(p)})+\bu^{(0)}\cdot\bn (E^{0}+p^{(0)})\bigg )\; d\Gamma\\
 \sum_{\sigma\in K} \big ( \rho^{(p)}_\sigma \big )^{-\kappa}\big ( (\Phi_p)^K_\sigma&+(r_p^K)_\sigma\big )+\sum_{\sigma\in K}p_\sigma^{(p)}
 \widetilde{\rho^\alpha}_\sigma (\Phi_\rho^K)_\sigma\\&=\int_K \big ( S^{(p)}-S^{(0)}\big ) \; d\bx+\oint_{\partial K}\dfrac{\bu^{(p)}\cdot\bn S^{(p)}+\bu^{(0)}\cdot\bn S^{(0)}}{2}\; d\Gamma
 \end{split}
 \end{equation*}
 This leads to a linear system of the type
 \begin{equation*}
 \begin{split}
 \sum_{\sigma\in K} (r_p^K)_\sigma&=\mathcal{E}_1\\
 \sum_{\sigma\in K}\big ( \rho^{(p)}_\sigma \big )^{-\kappa}(r_p^K)_\sigma&=\mathcal{E}_2
 \end{split}
 \end{equation*}
 where $\mathcal{E}_1$ and $\mathcal{E}_2$ are computable quantities. Since there is more than two degrees of freedom in $K$, if the density is not uniform, we can compute the corrections.  If the densities are the same ($\rho^{(p)}_\sigma=\rho^{(p)}$), then one has to choose initial residuals such that $\big ( \rho^{(p)} \big )^{-\kappa}\mathcal{E}_1=\mathcal{E}_2$.
 
 \begin{remark} Using this entropy, we keep the explicit nature of the scheme. With other entropies, this is less clear.
 \end{remark}

\section{Conclusions and perspectives}
In this paper, we have discussed how the conservation property writes in the residual distribution framework.  Instead of looking at what happens  at the cell interfaces, we look at the element contributions.  Using this concept, it is possible to reformulate most if not all the known schemes as finite volume schemes, with explicit formula for the flux. Using this notion, it is possible to construct schemes that start from a non conservative formulation of a conservative systems, or to enforce more than one conservation relation. We show the principles, and provide some examples. Other examples are  possible, see for example \cite{svetlana}.
\bibliographystyle{unsrt}
\bibliography{papier}
\end{document}